\documentclass[11pt,a4paper]{article}
\usepackage[utf8]{inputenc}
\usepackage[english]{babel}
\usepackage[T1]{fontenc}
\usepackage{a4wide}

\usepackage{float}
\usepackage{graphicx}
\usepackage{tikz}
\usepackage{epic}
\usetikzlibrary{patterns}
\usetikzlibrary{matrix,arrows,decorations.pathmorphing}
\usetikzlibrary{shapes.geometric}
\usetikzlibrary{svg.path}
\usepackage{pgf} 

\usepackage{amsmath, latexsym, amsfonts, amssymb, amsthm, amscd}
\usepackage{yfonts}
\usepackage{multirow}
\usepackage{mathdots}
\usepackage{dsfont}

\newtheorem{theorem}{Theorem}
\newtheorem{lem}{Lemma}
\newtheorem{defi}{Definition}
\newtheorem{prop}{Proposition}

\theoremstyle{definition}
\newtheorem{rem}{Remark}[section]

\newcommand{\EE}{\ensuremath{\mathbb{E}}}
\newcommand{\R}{\ensuremath{\mathbb{R}}}
\newcommand{\C}{\ensuremath{\mathbb{C}}}
\newcommand{\Z}{\ensuremath{\mathbb{Z}}}

\DeclareMathOperator*{\sgn}{sgn}
\providecommand{\keywords}[1]{\textbf{\textit{Keywords---}} #1}
\begin{document}
\renewcommand{\labelitemi}{\textbullet}

\title{A phase transition for $q$-TASEP with a few slower particles}

\author{Guillaume Barraquand\thanks{Univ Paris Diderot, Sorbonne Paris Cit\'e, Laboratoire de Probabilités et Modèles Aléatoires UMR 7599, 5 rue Thomas Mann, 75013 PARIS. E-mail: {\tt barraquand@math.univ-paris-diderot.fr}}}
\date{}

\maketitle

\begin{abstract}
We consider a $q$-TASEP model started from step initial condition where all but finitely many particles have speed $1$ and a few particles are slower. It is shown in \cite{ferrari2013tracy} that  the rescaled  particles position of $q$-TASEP with identical hopping rates obeys a limit theorem \`a la Tracy-Widom. We  adapt this work to the case of different hopping rates and show that one observes the so-called BBP transition. Our proof is a refinement of Ferrari-Vet\H o's and does not require any condition on the parameter $q$ nor the macroscopic position of particles.
\end{abstract}
\keywords{Interacting particle systems, KPZ universality class, Tracy-Widom distribution, phase transition.}

\section{Introduction and main result}

The totally asymmetric simple exclusion process is a stochastic model of particles on the lattice $\Z$, with at most one particle per site (exclusion principle). Each particle has an independent exponential clock and jumps to the right by one when it rings, provided the neighbouring site is empty. The $q$-TASEP is a generalization introduced in \cite{borodin2011macdonald}. In this model, the $i$-th particle jumps to the right by one  at rate $a_i(1-q^{\text{gap}})$ where the gap is the number of consecutive empty sites to the next particle on the right. The parameter $q\in(0,1)$ can hence be seen as a repulsion strength between particles. 

Among many other stochastic models, the TASEP and the $q$-TASEP lie in the KPZ universality class, named from the Kardar-Parisi-Zhang stochastic  partial differential equation modelling the growth of interfaces. The most prominent common features of models belonging to this class are fluctuations on a scale $t^{1/3} $,  spatial correlations on a scale $t^{2/3}$, and Tracy-Widom type  statistics (see the review \cite{corwin2012kardar} on the KPZ class). 

Borodin-Corwin's theory of Macdonald processes \cite{borodin2011macdonald} provides an algebraic framework to study integrable models in the KPZ universality class, extending the Schur processes \cite{okounkov2001infinite,okounkov2003correlation} which prove useful in the study of TASEP and other models. Macdonald processes are a two parameter family of measures on interlacing sequences of integer partitions, or Gelfand-Tsetlin patterns. The probability of a given pattern is expressed as a product of Macdonald functions, which are symmetric functions depending on two formal parameters $q$ and $t$. Various particular choices of the parameters are examined in \cite{borodin2011macdonald}, leading to applications to several stochastic models from statistical mechanics. When the parameter $t$ is set to zero, Macdonald functions degenerate to $q$-Whittaker functions. Besides the potential interest of the model itself, the introduction of $q$-TASEP is natural since it is a marginal of the $q$-Whittaker process in the same way as the TASEP is a marginal of the Schur process. 
When $q$ tends to $1$, the $q$-Whittaker functions become Whittaker functions whose connections with directed polymers were established in \cite{o2012directed}. 
The limit of the $q$-Whittaker process when $q$ goes to $1$ is investigated in \cite{borodin2011macdonald, borodin2012free}. Hence, $q$-TASEP interpolates between TASEP when $q=0$, and the O'Connell-Yor semi-discrete polymer when $q\to 1$, under a particular scaling of the parameters \cite{borodin2012free}. After further rescaling the space, the semi-discrete polymer model converges to the continuum directed random polymer. The predictions of the KPZ universality class about the scaling exponents and the  limit theorem towards the Tracy-Widom distribution are proved in \cite{johansson2000shape} for TASEP and in \cite{ borodin2012free} for the one-point distribution of the free energy of continuous polymers.
As for the $q$-TASEP, when all particles have the same speed $a_i\equiv 1$ and starting from the so-called step initial condition, Ferrari and Vet\H o \cite{ferrari2013tracy} show that the properly rescaled position of particles converges in law to the emblematic GUE Tracy-Widom distribution. Moreover, they confirm the KPZ scaling theory \cite{spohn2012kpz, krug1992amplitude}, which conjecturally predicts the exact value of all non-universal constants arising in the law of large numbers and the Tracy-Widom limit theorem. 

Another ubiquitous probability distribution appearing in the KPZ universality class is a generalization of the Tracy-Widom distribution called the BBP distribution. It first arose in spiked random matrix theory \cite{baik2005phase}. A phase transition happens for perturbed Gaussian  ensembles of hermitian matrices, and the fluctuations at the edge of the spectrum lie in the Gaussian regime or in the Tracy-Widom regime, according to the rank and the structure of the perturbation.  Using connections between sample covariance matrices and last passage percolation (LPP), Baik, Ben Arous and P\'ech\'e \cite{baik2005phase} explain how their results translate in terms of last passage percolation on the first quadrant where the first finitely many rows have different means. Using then a coupling between LPP  and TASEP, Baik \cite{baik2006painleve} explains that one also  observes the BBP transition for the fluctuations of the current in TASEP started from step initial condition, where all but finitely many particles have rate 1 and a few  have a smaller rate. The number of slower particles plays the same role as the rank of the perturbation in the matrix model.
On the other degeneration, 
that is when $q$ goes to $1$, Borodin, Corwin and Ferrari \cite{borodin2012free} show the same phase transition for the fluctuations of the free energy of the O'Connell-Yor directed polymer when adding a local perturbation of the environment in a quite similar way.

A finite number of slower particles in the $q$-TASEP creates a shock on a macroscopic scale. Our purpose is to show the same phase transition for the fluctuations of particles around their hydrodynamic limit, depending on  the minimum rate and the number of slower particles.
We adapt 
 the asymptotic analysis made in \cite{ferrari2013tracy}. 
 The main technical difference is the following: in \cite{ferrari2013tracy}, the fluctuation result is proved with a slightly restrictive condition on the macroscopic position of particles, which forbids to study the very first particles. It concerns a $ \mathcal{O}(\text{time})$ quantity of particles though. In our work, we do not assume this condition to hold, confirming that it was purely technical as suspected by Ferrari and Vet\H o.

\subsection{The $q$-TASEP}

The $q$-TASEP is a continuous-time Markov process described by the coordinates of particles $X_N(t)\in \Z, N\in\mathbb{N}^*, t\in\R_+$. Its infinitesimal generator  is given by 
\begin{equation}
\left( L f \right)(\textbf{x}) = \sum_{k\in\mathbb{N}^*} a_k (1-q^{x_{k-1}-x_k-1})(f(\textbf{x}^k) - f(\textbf{x}))
\end{equation}
where $\textbf{x} = (x_0, x_1, \dots)$ is such that $x_i > x_{i+1}$ for all $i$, $\textbf{x}^k$ is the configuration where $x_k$ is  increased by one, and by convention $x_0=+\infty$. The step initial condition corresponds to 
$\forall i\in\mathbb{N}^*, x_i(0)=-i$. 

Let us first focus on the case where all particles have equal hopping rates, $a_i\equiv 1$. In this case, the gaps between particles $(x_i-x_{i+1} -1)_i$ have the same dynamics as a $q$-TAZRP ($q$-deformed Totally Asymmetric Zero-Range Process) introduced in \cite{sasamoto1998exact},  for which general results on invariant distributions of zero-range processes apply \cite{andjel1982invariant}. Hence, \cite{borodin2011macdonald} shows that translation invariant extremal invariant measures are renewal processes $\mu_{r }$ for $r\in[0,1)$ on $\Z$, with renewal measure according to the $q$-Geometric distribution of parameter $r$, i.e.
\begin{equation*}
\mu_{r}(x_i-x_{i+1} -1=k) = (r, q)_{\infty}  \frac{r^k}{(q,q)_k}
\end{equation*}
where $(z, q)_k = (1-z)(1-qz)\dots(1-q^{k-1}z)$.

One expects that the rescaled particle density $\rho(x, \tau)$, given heuristically by 
\begin{equation*}
\rho(x, \tau) = \lim_{t\to\infty} \mathbb{P}(\text{ There is a particle at }\lfloor xt\rfloor \text{ at time }t\tau ),
\end{equation*}
 satisfies the PDE
\begin{equation}
\frac{\partial }{\partial \tau}\rho(x, \tau) + \frac{\partial}{\partial x}j(\rho)(x, \tau) = 0
\label{equationparticleconservationPDE}
\end{equation}
where $j(\rho)$  is the particle current at density $\rho$.

By local stationarity assumption, we mean that gaps between particles are locally distributed according to i.i.d. $q$-geometric random variables for some parameter $r$ which depends on the macroscopic position. Under this assumption and using the particle conservation PDE (\ref{equationparticleconservationPDE}), one can guess the deterministic profile of the $q$-TASEP, see \cite[Section 3]{ferrari2013tracy} for details. 
In order to state this result, we need some preliminary definitions. 
We fix the parameter $q\in (0,1)$ and choose a real number $\theta >0$ which parametrizes the macroscopic position of particles, as explained below. 
\begin{defi}[\cite{ferrari2013tracy}]
We recall the definition of the $q$-Gamma function
\begin{equation*}
\Gamma_q(z)=(1-q)^{1-z}\frac{(q;q)_\infty}{(q^z;q)_\infty}.
\end{equation*}
Then the $q$-digamma function is defined by
\begin{equation*}
\Psi_q(z)=\frac\partial{\partial z}\log \Gamma_q(z).
\end{equation*}
For $q\in (0,1)$ and $\theta >0$, we also define the functions
\begin{eqnarray}
\kappa\equiv & \kappa(q,\theta)=&\frac{\Psi'_q(\theta)}{(\log q)^2q^\theta},\label{defkappa} \\
f\equiv & f(q,\theta)=&\frac{\Psi'_q(\theta)}{(\log q)^2}-\frac{\Psi_q(\theta)}{\log q}-\frac{\log(1-q)}{\log q},\label{deff}\\
\chi\equiv &\chi(q,\theta)=&\frac{\Psi'_q(\theta)\log q-\Psi''_q(\theta)}2\label{defchi}.
\end{eqnarray}
\end{defi}
Then, the following law of large numbers holds when $N$ goes to infinity,
\begin{equation*}
\frac{X_N(\tau = \kappa N )}{N}\longrightarrow f-1.
\end{equation*}
Let us explain the arguments leading to this result. Under local stationarity assumption, for a real number $x$ such that around $xt$, gaps between particles are distributed as $q$-Geometric random variables of parameter $r$, we have 
\begin{equation}
\rho(x, t) = \frac{1}{1+\EE[\mathrm{gap}]} = \frac{\log q}{\log(q) + \log( 1-q) + \Psi_q(\log_q r)}.
\end{equation}
Writing $x= x(r)$ and after the change of variables $\log_q r =\theta$, the PDE (\ref{equationparticleconservationPDE}) implies that $x(\theta) = \left(f(q, \theta)-1\right)/\kappa(q, \theta)$. Finally, integrating this density gives a law of large numbers for the integrated current of particles, or the equivalent statement on $X_N(\tau)$ given above.
Moreover, under local stationarity assumption, the gaps between consecutive particles around particle $N$ at time $\kappa(q, \theta) N$ are distributed according to i.i.d. $q$-Geometric random variables of parameter $q^{\theta}$. The averaged hopping rate (or the averaged speed of a tagged particle) is then $\EE[1-q^{\mathrm{gap}}] = q^{\theta}$.
\begin{rem}
One could choose the time $\tau$ to be the free parameter which tends to infinity, but the formulae are slightly simpler when $N$ tends to infinity and $\tau$ depends on $N$.
\end{rem}

\subsection{Main result}

In this paper, we aim to study a $q$-TASEP started form step initial condition where all but finitely many particles have rate $1$, and some particles are slower. Notice that as we will see in the proofs, the case of a finite number of faster particles does not change anything to the macroscopic behaviour.

 Let $a_{i_1}, a_{i_2}, \dots , a_{i_m}$ be the rates of the slower particles, and $\alpha$ be the rate of the slowest particle and suppose that $k\leqslant m$ particles have rate $\alpha$. For later use, it is convenient to set the notation $A:=\log_q(\alpha)$.
  
The slower particles create a shock on a macroscopic scale and influence the law of large numbers. The particles which are concerned by the shock are all the particles that, in absence of slower particles, would have an averaged hopping rate greater than $\alpha$. 
In order to state the results precisely, we need to define some additional functions.
\begin{defi}
For $q\in (0,1)$ and $\theta >0$, we define
\begin{eqnarray}
 g\equiv & g(q, \theta) = &\frac{\Psi'_q(\theta)}{(\log q)^2}\frac{\alpha}{q^{\theta}}-\frac{\Psi_q(A)}{\log q}-\frac{\log(1-q)}{\log q},\label{defg} \\
\sigma\equiv &\sigma(q,\theta)=&\Psi'_q(\theta) \frac{\alpha}{q^{\theta}}-\Psi'_q(A).\label{defsigma}
\end{eqnarray}
\end{defi}
Then the following law of large numbers holds, when $N$ goes to infinity,
\begin{equation*}
\frac{X_N(\tau=\kappa N)}{N} \longrightarrow \left\lbrace\begin{matrix}
f-1 & \text{ when } & \alpha > q^{\theta},\\
g-1 & \text{ when } & \alpha \leqslant q^{\theta}.\\
\end{matrix} \right. 
\end{equation*}
The limit shape of $\frac{1}{\tau}(X_N(\tau)+N, N) $ is drawn in Figure \ref{figurelimitshape}. One sees that the limiting profile is linear when $\alpha < q^{\theta}$. It means that the density of particles is constant inside the shock and particles have an averaged speed $ \alpha$. 

The scaling theory of single-step growth models explained in \cite{spohn2012kpz}  and the KPZ class conjecture give a way to compute the non universal constants arising in the law of large numbers, and the scale and precise variance of fluctuations. Nevertheless, these results are expected to hold only at points where the limiting shape is strictly convex. The results in \cite{ferrari2013tracy} confirm the conjecture for a $q$-TASEP where all particles have equal hopping rates. In the presence of slower particles, the limiting shape is linear inside the shock (cf Figure \ref{figurelimitshape}), and the fluctuations are not predicted by the KPZ class conjecture.

\begin{rem}
In \cite{spohn2012kpz}, the convexity condition is given on the limit profile of the height function, but this is equivalent.
\end{rem}
 
\begin{figure}
\begin{center}
\includegraphics[scale=0.7]{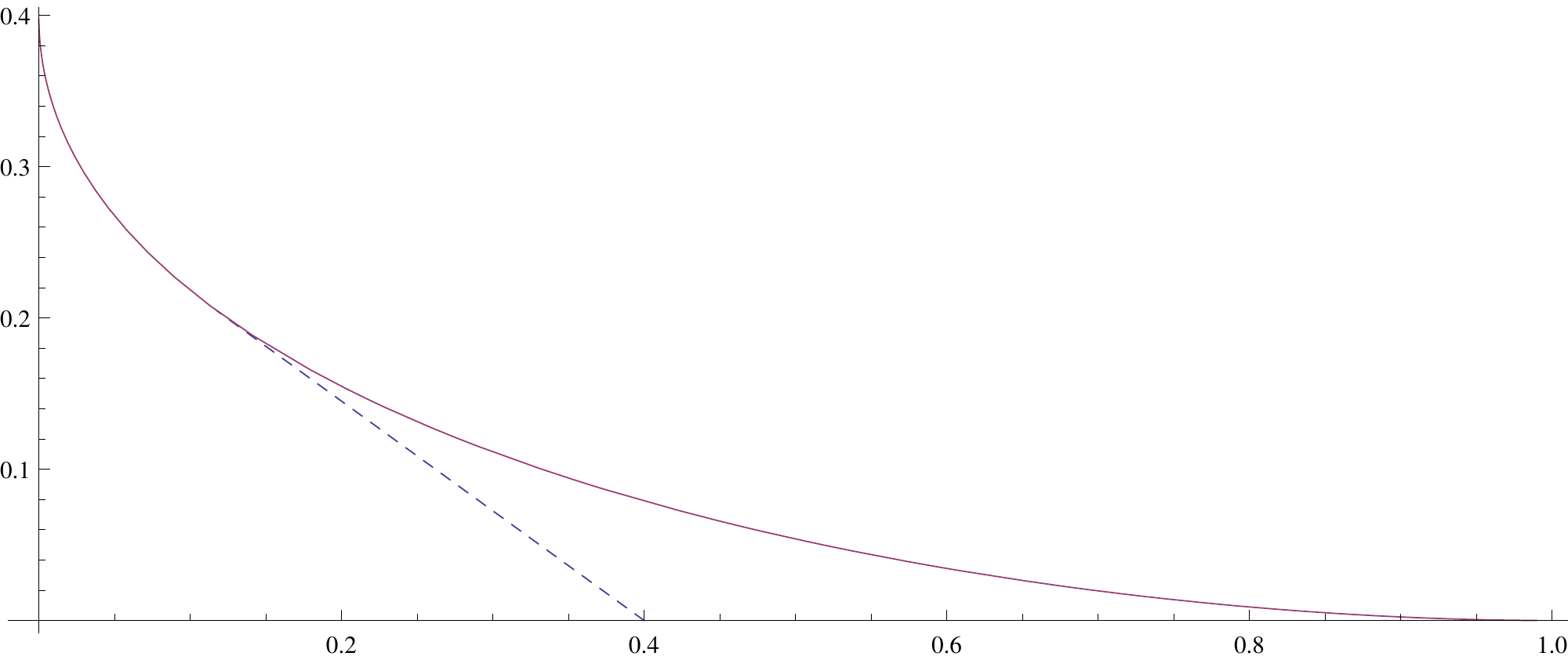}
\end{center}
\caption{Limit shapes of $\frac{1}{\tau}(X_N(\tau)+N, N) $ for $q=0.6$. The solid line corresponds to  $\alpha=1$ (no slow particle). The dashed line corresponds to $\alpha=0.4$ (one or several slower particles). Note that the curved line is the parametric curve $(f/\kappa, 1/\kappa)$, whereas the straight line part is the parametric curve $(g/\kappa, 1/\kappa)$, when $\theta$ ranges from $0$ to $\log_q \alpha $, i.e. $\theta$ satisfying the condition $\alpha<q^{\theta}$. }
\label{figurelimitshape}
\end{figure}

Let us explain the scalings that we use in the paper. When $\theta$ is such that $q^{\theta} <\alpha$, particles around the macroscopic position parametrized by $\theta$ have speed $q^{\theta} $ and are asymptotically independent from the slower particles. We expect to observe Tracy-Widom fluctuations on a $N^{1/3}$ scale with spatial correlation on a $N^{2/3}$ scale. Hence, the time $\tau(N,c)$ is set as 
\begin{equation*}
\tau(N,c)=\kappa N+ c q^{-\theta} N^{2/3},
\end{equation*}
where $c$ is a free but fixed parameter.
The macroscopic position of the $N^{\text{th}}$ particle,  denoted $p(N,c)$, is given by
\begin{equation*}
p(N,c) = (f-1)N +cN^{2/3} - c^2 \frac{(\log q)^3}{4\chi} N^{1/3}\  .
\end{equation*}
When $ \theta $  is such that $\alpha < q^{\theta}$, we expect the fluctuations to live on the $N^{1/2}$ scale, although the limiting law is not necessarily Gaussian. We set 
\begin{equation*}
\tau^*(N, c)=\kappa N+ c N^{1/2}/\alpha,  
\end{equation*}
and the macroscopic position is 
\begin{equation*} 
 p^*(N,c) =(g-1)N +cN^{1/2}.
\end{equation*}
The aim is to study the fluctuations of $X_N(\tau(N,c))$ (resp. $X_N(\tau^*(N,c))$) around $p(N,c)$ (resp. $p^*(N,c)$).

Next, we define classical probability distributions from random matrix theory in a convenient way for our purposes.
\begin{defi}[Distribution functions]
\label{defdistributions}
\begin{enumerate}
\item The distribution function $F_{\rm GUE}(x)$ of the GUE Tracy-Widom distribution is defined  by $F_{\rm GUE}(x)=\det(I-K_{\rm Ai})_{\mathbb{L}^2(x,+\infty )}$ where $K_{\rm Ai}$ is the Airy kernel, 
\begin{equation*}
K_{\rm Ai} (u, v) = \frac{1}{(2i\pi)^2} \int_{e^{-2i\pi/3}\infty}^{e^{2i\pi/3}\infty} \mathrm{d}w \int_{e^{-i\pi/3}\infty}^{e^{i\pi/3}\infty} \mathrm{d}z \frac{e^{z^3/3-zu}}{e^{w^3/3-wv}}\frac{1}{z-w}, 
\end{equation*}
where the contours for $z$ and $w$ do not intersect.
\item Let $\mathbf{b} = (b_1, \dots, b_k)\in \R^k$. The BBP  distribution of rank $k$ from \cite{baik2005phase} is defined by   $F_{\rm BBP,k, \mathbf{b}}(x) =\det(I-K_{\rm BBP, k, \mathbf{b}})_{\mathbb{L}^2(x,+\infty )}$ where $K_{\rm BBP, k, \mathbf{b}}$ is given by
\begin{equation*}
K_{\rm BBP, k, \mathbf{b}} (u, v) = \frac{1}{(2i\pi)^2} \int_{e^{-2i\pi/3}\infty}^{e^{2i\pi/3}\infty} \mathrm{d}w \int_{e^{-i\pi/3}\infty}^{e^{i\pi/3}\infty} \mathrm{d}z \frac{e^{z^3/3-zu}}{e^{w^3/3-wv}}\frac{1}{z-w}\left(\frac{z-b}{w-b}\right)^k,
\end{equation*}
where the contour for $w$ passes to the right of the $b_i$'s, and the contours for $z$ and $w$ do not intersect.
\item $G_k(x)$ is the distribution of the largest eigenvalue of a $k\times k$ GUE, which also has a Fredholm determinant representation. $G_k(x)=\det(I-H_k)_{\mathbb{L}^2(x,+\infty )}$, where $H_k$ is the Hermite kernel given by
\begin{equation*}
H_k(u, v) = \frac{c_{k-1}}{c_k} \frac{p_k(u)p_{k-1}(v) - p_{k-1}(u)p_k(v)}{u-v} e^{-(u^2+v^2)/4}, 
\end{equation*}
where $c_n = 1/((2\pi)^{1/4}\sqrt{n!})$ and $(p_n)_{n\geqslant 0}$ is a family of orthogonal polynomials determined by $\int_{-\infty}^{\infty}  p_m(t)p_n(t) e^{-t^2/2}\mathrm{d}t = \delta_{mn}$. The kernel $H_k$ has an integral representation
\begin{equation*}
H_k(u, v) = \frac{1}{(2i\pi)^2} \int_{e^{i(\varphi - \pi)}\infty}^{e^{i(\pi-\varphi)}\infty} \mathrm{d}w \int_{e^{i(\gamma-\pi/2)}\infty}^{e^{i(\pi/2-\gamma)}\infty} \mathrm{d}z \frac{e^{z^2/2-zu}}{e^{w^2/2-wv}}\frac{1}{z-w}\left(\frac{z}{w}\right)^k , 
\end{equation*}
with $\varphi, \gamma \in(0, \pi/4)$, and where the contour for $w$ passes to the right of $0$, and the contours do not intersect. For the equivalence between these formulas, see e.g.  \cite{bleher2005integral} and references therein.
\end{enumerate}
\end{defi}
We are now able to state our main result.
\begin{theorem}
\label{theomainresult}
Let $k$ be the number of particles having rate $\alpha$.
\begin{itemize}
\item If $q^{\theta}<\alpha$, then writing $X_N(\tau(N))=(f-1)N +cN^{2/3} - c^2 \frac{(\log q)^3}{4\chi} N^{1/3}  +\frac{\chi^{1/3}}{\log q}\xi_NN^{1/3}$, we have for any $x\in\R$, 
\begin{equation*}
\lim_{N\to\infty}\mathbb{P}(\xi_N <x)=F_{\rm GUE}(x).
\end{equation*}
\item If $q^{\theta}=\alpha$ then writing again $X_N(\tau(N))=(f-1)N +cN^{2/3} - c^2 \frac{(\log q)^3}{4\chi} N^{1/3}  +\frac{\chi^{1/3}}{\log q}\xi_NN^{1/3}$, we have for any $x\in\R$, 
\begin{equation*}
\lim_{N\to\infty}\mathbb{P}(\xi_N <x)=F_{\rm BBP,k, \mathbf{b}}(x)
\end{equation*}
where $\mathbf{b} = (b, \dots, b)$ with $b=  \frac{c(\log q)^2}{2\chi^{2/3}} $.
\item If $q^{\theta}>\alpha$, then writing $X_N(\tau^*(N))=(g-1)N +cN^{1/2}+\frac{\sigma^{1/2}}{\log q}\xi_NN^{1/3}$, we have for any $x\in\R$,
\begin{equation*}
\lim_{N\to\infty}\mathbb{P}(\xi_N <x)=G_k(x).
\end{equation*}
\end{itemize}
\end{theorem}
\begin{rem}
These results on the asymptotic positions of particles readily translate in terms of current of particles or height function, as in \cite[Theorem 2.9]{ferrari2013tracy}.
\end{rem}
\begin{rem}
Furthermore, it is possible to observe the BBP distribution $F_{\mathrm{BBP}, k, \mathbf{b}}$ for any arbitrary vector $\mathbf{b}$. One has to perturb the rates on a scale $N^{-1/3}$. Let fix some $\theta >0$ and assume that for $1\leqslant i\leqslant k$, $a_i = q^{ \theta + \tilde{b}_i N^{-1/3}}$ and the rates of all other particles are higher than $q^{\theta} $. Then, if the random variable $\xi_N $ is defined as in Theorem \ref{theomainresult}, $\lim_{N\to\infty}\mathbb{P}(\xi_N <x)=F_{\rm BBP,k, \mathbf{b}}(x)$ where  
$\mathbf{b} = (b + \tilde{b}_1, \dots , b + \tilde{b}_1) $ with $b=  \frac{c(\log q)^2}{2\chi^{2/3}} $ as before. In terms of Macdonald processes, this perturbation of the rates corresponds to a perturbation of the parameters of the $q$-Whittaker process. It is the precise analogue of the perturbation of the parameters of the Whittaker process applied to the O'Connell-Yor semi-discrete random polymer, for which the same result holds for the fluctuations of the free energy  \cite[Theorem 2.1]{borodin2012free}.
\label{rem:fullBBP}
\end{rem}
\section{Asymptotic analysis}

In this section, we prove the main theorem \ref{theomainresult}. The asymptotic analysis we present here is an instance of Laplace's method closely adapted from \cite{ferrari2013tracy}. Fix a $\theta >0$ and the parameter $c\in\R$. We start from a Fredholm determinant representation for the $q$-Laplace transform of $q^{X_N(\tau)}$, which  characterizes the law of $X_N(\tau)$. It was first proved in \cite{borodin2011macdonald, borodin2012duality} in a slightly different form, and in \cite{borodin2012free} in the form which seems the most convenient for an  asymptotic analysis. Before stating this result, we  need  to define some integration contours in the complex plane.
\begin{defi}[see Figure \ref{figurecontours1}]
\label{defcontour1}
We define a family of contours  $\tilde{\mathcal{C}}_{\bar{\alpha}, \varphi}$ for $\bar{\alpha}>0 $ and $\varphi\in(0, \pi/2) $ by 
\begin{equation*}
\tilde{\mathcal{C}}_{\bar{\alpha}, \varphi} = \lbrace \bar{\alpha} + e^{i\varphi \sgn(y)}\vert y \vert , y\in\R \rbrace,
\end{equation*}
 oriented from bottom to top. For every $w\in \tilde{\mathcal{C}}_{\bar{\alpha}, \varphi}$, we define a contour $\tilde{\mathcal{D}}_w$ by 
 \begin{multline*}
 \tilde{\mathcal{D}}_w = \left]R-i\infty, R-id \right] \cup \left] R-id, 1/2 -id\right] \cup \left] 1/2 -id, 1/2 +id\right]\\
  \cup \left] 1/2 +id , R+id\right] \cup \left]R+id , R+i\infty \right[ 
 \end{multline*}
 oriented from bottom to top, where $R,d>0$ are chosen such that:
\begin{itemize}
\item[(i)]Let $b=\pi/4 -\varphi/2$, then $\arg(w(q^s-1))\in(\pi/2+b, 3\pi/2-b)$ for any $s\in \tilde{\mathcal{D}}_w$.
\item[(ii)]$q^sw, s\in \tilde{\mathcal{D}}_w$ stays to the left of $\tilde{\mathcal{C}}_{\bar{\alpha}, \varphi}$.
\end{itemize}
These contours always exist (see Remark 4.9 in \cite{borodin2012free}): to check condition (ii), it is enough to prove that the points $wq^{1/2\pm ir}$ stay on the left of $\tilde{\mathcal{C}}_{\bar{\alpha}, \varphi}$ which follows from simple geometric arguments for $d$ small enough. It can be seen in Figure  \ref{figurecontours1}. To satisfy condition (i), the argument of $w q^s- w$ can be made as small as we want choosing $d$ small enough and $R$ large enough.
\end{defi} 
\begin{figure}
\begin{center}
\begin{tikzpicture}[scale=1]
\draw[thick] (-1, 0) -- (6,0);
\draw (0,0.1) -- (0,-0.1) node[anchor=north]{0};
\draw (0,0) circle (0.6);
\fill[black] (20:5.1) circle(0.1);
\draw[dashed] (-0,0) -- (20:5);
\draw (24:0.6) -- (24:4);
\draw (16:0.6) -- (16:4);
\draw (24:4) -- (16:4);

\draw (3,0)-- (27:7);
\draw (3,0) -- (-27:7);
\draw  (4,0) arc(0:43:1);

\draw  (20:5.1) node[anchor=north west]{$w$} ;
\draw    (-26:6) node [anchor=south west] {$\tilde{\mathcal{C}}_{\bar{\alpha}, \varphi}$};
\draw (3,0) node[anchor = north]{$\bar{\alpha}$};
\draw (4,0) node [anchor=south west]  {$\varphi$};
\draw (1,1.5) node{$q^s w , s\in\tilde{\mathcal{D}}_w$};
\end{tikzpicture}
\end{center}
\caption{The contours in Definition \ref{defcontour1}. One sees that $q^sw, s\in \tilde{\mathcal{D}}_w$ stays to the left of $\tilde{\mathcal{C}}_{\bar{\alpha}, \varphi}$. }
\label{figurecontours1}
\end{figure}
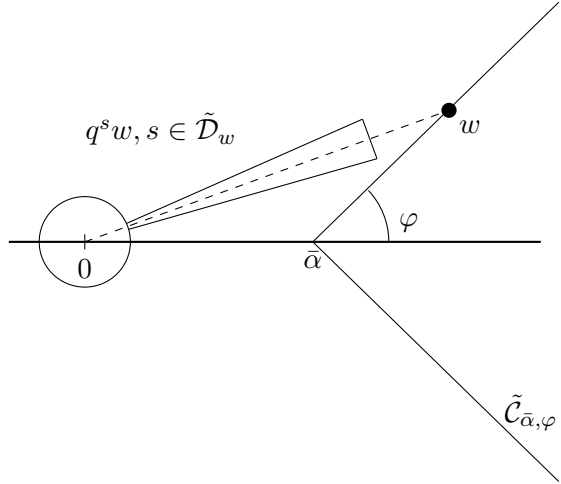

\begin{theorem}
\label{theofredholmdet}
Let $\zeta\in\C\setminus \R_+$ and $N\geqslant \max_{1\leqslant j\leqslant m}\lbrace i_j\rbrace$. Then, 
\begin{equation}
\EE\left[\frac{1}{(\zeta q^{X_N(t)+N} ; q)_{\infty}}\right] = \det (I+\tilde{K}_{\zeta})_{\mathbb{L}^2(\tilde{\mathcal{C}}_{ \bar{\alpha},\varphi})}
\label{equationfredholmdet}
\end{equation}
for any $\varphi\in (0,  \pi/4)$ provided the condition $0<\bar{\alpha}<\alpha$ is satisfied. The operator $\tilde{K}_{\zeta}$ is defined by its integral kernel 
\begin{equation*}
\tilde{K}_{\zeta}(w,w') = \frac{1}{2i\pi} \int_{\tilde{\mathcal{D}}_w} \mathrm{d} s \Gamma(-s)\Gamma(1+s) \left(-\zeta  \right)^s g_{w,w'}(q^s), 
\end{equation*}
where 
\begin{equation*}
g_{w,w'}(q^s) = \frac{\exp(tw(q^s-1))}{q^sw - w'} \left(\frac{(q^s w ; q)_{\infty}}{( w ; q)_{\infty}} \right)^{N-m} \prod_{1\leqslant j\leqslant m} \left(\frac{( q^s w/a_{i_j} ; q)_{\infty}}{( w/a_{i_j} ; q)_{\infty}} \right).
\end{equation*}
\end{theorem}
\begin{rem}
The condition $0<\bar{\alpha}<\alpha$  ensures that  all the poles for the variable $w$ in $g_{w,w'}(q^s) $ are inside the contour $\tilde{\mathcal{C}}_{ \bar{\alpha},\varphi}$.
\end{rem}
\begin{proof}
Let us  explain how this theorem is a rephrasing of a known result on Macdonald processes.
The (ascending) Macdonald processes introduced in \cite{borodin2011macdonald} are a family of probability measures on sequences of integer partitions $\lambda^1, \lambda^2, \dots, \lambda^N$,  
 where $\lambda^i$ has at most $i$ non-zero components, and the sequence satisfies the  interlacing condition $ \lambda_{j+1}^k\leqslant \lambda_j^{k-1} \leqslant \lambda_j^k$, for all $1\leqslant j\leqslant k-1<N$. The Macdonald measures are a family of measures on integer partitions such that the marginals of Macdonald processes are Macdonald measures cf \cite[paragraph 2.2.2]{borodin2011macdonald}.

The probability of a given configuration is expressed as a product of Macdonald functions, which are symmetric functions in infinitely many variables, such that the coefficient of each monomial is a rational function in two parameters $q$ and $t$. In order to build a genuine positive measure, one has to properly specialize Macdonald functions and consider $q$ and $t$ as real parameters. 
 
 Different choices for the parameters $q$ and $t$ are examined in \cite{borodin2011macdonald}.   When $t=0$, the study of the dynamics preserving the Macdonald process leads to the definition of the $q$-TASEP. Indeed, under a  specialization of the Macdonald process with parameters $(a_1, \dots ,a_N ; \rho_{\tau})$ where $\rho_{\tau}$ is the Plancherel specialization of parameter $\tau$, the marginal $(\lambda_k^k)_{1\leqslant k \leqslant N}$ has Markovian dynamics which is exactly that of the $q$-TASEP. 
More precisely, $(X_k(\tau)+k)_{1\leqslant k\leqslant N}$ has the same distribution as the marginal $(\lambda_k^k)_{1\leqslant k \leqslant N}$ for any time $\tau$. 
  The theorem is now a reformulation of Theorem 4.13 in \cite{borodin2012free} which proves the same formula for $\EE\left[1/(\zeta q^{\lambda_N^N} ; q)_{\infty}\right]$ when $\lambda^N$ is distributed according to the Macdonald measure under specialization $(a_1, \dots ,a_N ; \rho_{\tau})$.
  
\end{proof}
We make the change of variables: 
\begin{equation*}
w=q^W, \ \ \ w' = q^W, \ \ s+W=Z.
\end{equation*}
This demands to introduce new integration contours for the variables $Z$, $W$ and $W'$, depicted in Figure \ref{figurecontours2}.
Let $\bar{A} = \log_q(\alpha)$ and let $\mathcal{C}_{\bar{A}, \varphi}$ be the image of $\tilde{\mathcal{C}}_{\bar{\alpha}, \varphi}$ under the map $x\mapsto \log_q x$. 
For every $W\in \mathcal{C}_{\bar{A}, \varphi}$, let  $\mathcal{D}_W = \lbrace\bar{A}+\sigma +i\R \rbrace \cup \mathcal{E}_1 \cup \dots \cup \mathcal{E}_{k_W}$, the value of $\sigma>0 $ to be chosen later\footnote{Note that the real number $\sigma$ that we use here has nothing to do with the standard deviation $\sigma$ in Equation \ref{defsigma}, but we allow this abuse of notations to keep the same notations as in \cite{ferrari2013tracy}},  where $\mathcal{E}_1, \dots, \mathcal{E}_{k_W}$  are  small circles around the residues coming from the sine at $W+1, W+2, \dots, W+k_W$.

More precisely, the vertical line is modified in a neighbourhood of size $\delta$ around the real axis as in Figure \ref{figurecontours2}, and we choose $\sigma $ such that the poles coming from the sine  inverse are at a distance from the vertical line at least $\sigma/2$. To make this possible, the vertical lines of the contour are chosen to have real part $ \bar{A} +\sigma$ or $ \bar{A} +2\sigma$.
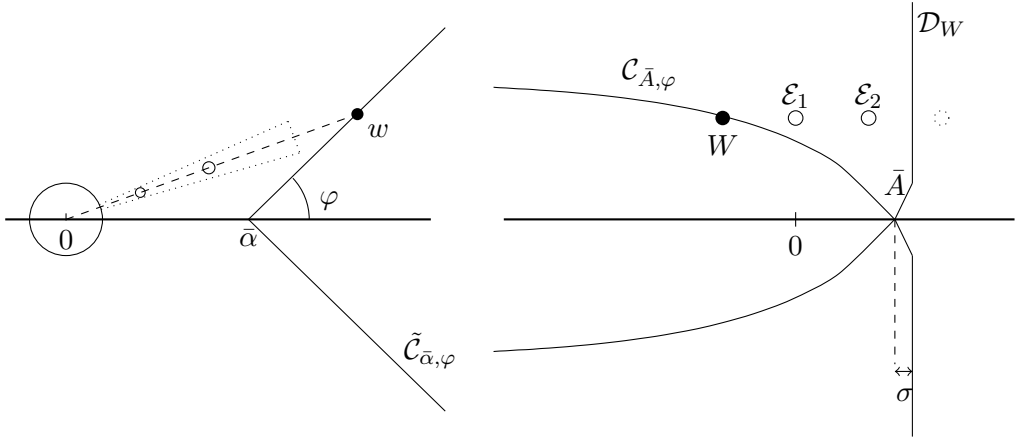
\begin{figure}
\begin{tikzpicture}[scale=0.8]
\begin{scope}[ scale=1]
\draw[thick] (-1, 0) -- (6,0);
\draw (0,0.1) -- (0,-0.02) node[anchor=north]{$0$};
\draw (0,0) circle (0.6);
\fill[black] (20:5.1) circle(0.1);
\draw[dashed] (-0,0) -- (20:5);
\draw[dotted] (24:0.6) -- (24:4);
\draw[dotted] (16:0.6) -- (16:4);
\draw[dotted] (24:4) -- (16:4);
\draw (20: 1.3) circle(0.08);
\draw (20: 2.5) circle(0.1);

\draw (3,0)-- (27:7);
\draw (3,0) -- (-27:7);
\draw  (4,0) arc(0:43:1);

\draw  (20:5.1) node[anchor=north west]{$w$} ;
\draw    (-26:6) node [anchor=south west] {$\tilde{\mathcal{C}}_{\bar{\alpha}, \varphi}$};
\draw (3,0) node[anchor = north]{$\bar{\alpha}$};
\draw (4,0) node [anchor=south west]  {$\varphi$};
\end{scope}
\begin{scope}[xshift=12cm, scale=1.2]
\draw[thick] (-4, 0) -- (3,0);
\draw (0,0.1) -- (0,-0.1) node[anchor=north]{$0$};
\draw[scale=1, domain=0:8, smooth, variable=\t] plot({1/2*ln((0.5+cos(43 r)*\t )^2 +(sin(43 r)*\t)^2)/ln(0.6)},{2/ln(0.6)*pi/180*atan(sin(43 r)*\t/(0.5+cos(43 r)*\t+sqrt((0.5+cos(43 r)*\t )^2 +(sin(43 r)*\t)^2) )) });
\draw[scale=1, domain=0:8, smooth, variable=\t] plot({1/2*ln((0.5+cos(43 r)*\t )^2 +(sin(43 r)*\t)^2)/ln(0.6)},{-2/ln(0.6)*pi/180*atan(sin(43 r)*\t/(0.5+cos(43 r)*\t+sqrt((0.5+cos(43 r)*\t )^2 +(sin(43 r)*\t)^2) )) });
\fill (-1, 1.4) circle(0.1);
\draw (0, 1.4) circle(0.1);
\draw (1, 1.4) circle(0.1);
\draw[dotted] (2, 1.4) circle(0.1);
\draw (1.36,0) -- (1.6,0.5);
\draw (1.36,0) -- (1.6,-0.5);
\draw (1.6, 0.5) -- (1.6, 3);
\draw (1.6,- 0.5) -- (1.6, -3);
\draw[dashed] (1.36,0) -- (1.36,-2);
\draw[<->](1.36,-2.1) -- (1.6, -2.1);

\draw (2,2.7) node{$\mathcal{D}_W $};
\draw (-1, 1.3) node[anchor = north] {$W$};
\draw (0, 1.4) node[anchor = south] {$\mathcal{E}_1$};
\draw (1, 1.4) node[anchor = south] {$\mathcal{E}_2$};
\draw (1.36,0.2) node[anchor=south] {$\bar{A}$};
\draw (-2, 2) node{$\mathcal{C}_{\bar{A}, \varphi} $};
\draw (1.5, -2.2) node[anchor=north] {$\sigma$};
\end{scope}
\end{tikzpicture}
\caption{The contours for the variables $Z$, $W$ and $W'$ are on the right-hand-side. The left-hand-side is the image by the map $Z\mapsto q^Z$. In this example, $k_W=2$.}
\label{figurecontours2}
\end{figure}

We obtain, as in \cite{ferrari2013tracy}, the kernel
\begin{multline*}
\hat{K}_{\zeta}(W,W') = \frac{q^W \log q}{2i\pi} \int_{\mathcal{D}_W} \frac{\mathrm{d}Z}{q^Z - q^{W'}} \frac{\pi}{\sin(\pi(W-Z))}\\
\frac{(-\zeta)^Z \exp(\tau q^Z + (N-m) \log(q^Z; q)_{\infty} +\sum_{j=1}^m \log(q^Z/a_{i_j}; q)_{\infty} )}{(-\zeta)^W \exp(\tau q^W + (N-m) \log(q^W; q)_{\infty})+\sum_{j=1}^m \log(q^W/a_{i_j}; q)_{\infty} )} .
\end{multline*}

\subsection{Case $\alpha > q^{\theta}$, Tracy-Widom fluctuations.}

Fix $ \theta >0$ such that  $\alpha > q^{\theta}$ and $c\in\R$. We want to study the limit of 
\begin{equation}
\mathbb{P}\left(\frac{X_N(\tau(N,c)) - p(N,c)}{\chi^{1/3}/(\log q) N^{1/3} } <x \right). 
\label{equationprobarecherchee}
\end{equation}
The function $ z\mapsto 1/(z ;q)_{\infty}$ converges uniformly as an infinite product for $z\in]-\infty, 0]$. Thus when  $z$ goes to $-\infty$, then $1/(z ;q)_{\infty}$ goes to zero, and when $z$ goes to zero then $1/(z ;q)_{\infty}$ goes to $1$. Modulo a justification of the exchange between expectation and limit that we explicit in the end of this subsection,  if
\begin{equation*}
\zeta = -q^{-p(N,c)-N-\frac{\chi^{1/3}}{\log q}xN^{1/3}} 
\end{equation*}
then (\ref{equationprobarecherchee}) and $\mathbb{E}\left[1/(\zeta q^{X_N(\tau)} ; q)_{\infty} \right]$ have the same limit. Thus, in the following of this subsection, we set $\zeta $ as above.
We fix also $\bar{\alpha}=q^{\theta}$ (or equivalently $\bar{A}=\theta$). As we assume $\alpha > q^{\theta}$, the condition $0<\bar{\alpha}<\alpha$ in Theorem \ref{theofredholmdet} is satisfied. 
We then obtain
$\det (I+\tilde{K}_{\zeta})_{\mathbb{L}^2(\tilde{\mathcal{C}}_{ \bar{\alpha},\varphi})} = \det(I+K_x)_{\mathbb{L}^2(\mathcal{C}_{ \theta,\varphi})}$ where 
\begin{multline}
K_x(W,W')= \\\frac{q^W \log q}{2i\pi} \int_{\mathcal{D}_W} \frac{\mathrm{d}Z}{q^Z - q^{W'}}\frac{\pi}{\sin(\pi(Z-W))} \frac{\exp(Nf_0(Z)+N^{2/3} f_1(Z) + N^{1/3}f_2(Z))}{\exp(Nf_0(W)+N^{2/3} f_1(W) +N^{1/3}f_2(W))}\frac{\phi(Z)}{\phi(W)},
\label{equationexpressionkernel}
\end{multline}
with
\begin{eqnarray*}
f_0(Z) &=& - f (\log q)Z  + \kappa q^Z + \log(q^Z; q)_{\infty},\\
f_1(Z) &=& -c (\log q) Z + cq^{Z-\theta},\\ 
f_2(Z) &=& c^2 \frac{(\log q)^4}{4\chi} Z-\chi^{1/3} x Z,\\
\phi(Z)&=& \frac{\prod_{j=1}^m (q^Z/a_{i_j} ; q)_{\infty}}{\left((q^Z ; q)_{\infty}\right)^m}.
\end{eqnarray*}

Let us describe the idea of Laplace's method in our context. The asymptotic behaviour of the kernel is governed by the variations of the real part of $f_0$. In the sequel, we exhibit steep-descent contours, which allows us to prove that in the large $N$ limit, the main contribution to the Fredholm determinant is localized in a neighbourhood of $\theta$ which is the critical point of $ \Re[f_0]$. Then, using estimates and Taylor expansions for the argument of the exponential inside the kernel, we prove the limit.
 Due to the difficulty to simultaneously find a steep-descent path for the contour of the Fredholm determinant and to control the extra residues coming from the sine inverse in formula (\ref{equationexpressionkernel}), the authors in \cite{ferrari2013tracy} impose a technical condition $q^{\theta} \leqslant \frac12$, suspecting that it was purely technical (see Remark 2.5).
In order to get rid of this condition, we do not choose exactly the same contours. Our contour for the variable $W$ is $\mathcal{C}_{\theta, \pi/4}$ instead of $\mathcal{C}_{\theta, \varphi}$ for $\varphi$ close to $\pi/2$. Indeed, for $\varphi \neq \pi/4$, the contour $\mathcal{C}_{\theta, \varphi}$ is not necessarily steep-descent for $-\Re[f_0]$ when  $q^{\theta}>1/2 $.

For later use, we give two useful series representations for $\Psi_q$ and its derivative
\begin{equation}
\Psi_q(Z) = -\log(1-q) + \log q \sum_{k=0}^{\infty} \frac{q^{Z+k}}{1-q^{Z+k}}, 
\label{equationseriespsi}
\end{equation}
\begin{equation}
\Psi'_q(Z) = (\log q)^2 \sum_{k=0}^{\infty} \frac{q^{Z+k}}{(1-q^{Z+k})^2}.
\label{equationseriespsiprime}
\end{equation}

In general, we parametrize the contour $\mathcal{C}_{\bar{A}, \varphi}$ by $W(s) = \log_q(\alpha+\vert s\vert e^{i\varphi \sgn{(s)}})$ for $s\in\R$. For instance, the contour $\mathcal{C}_{\theta, \pi/4}$ is parametrized by $W(s) = \log_q(q^{\theta}+\vert s\vert e^{\sgn{(s)}i \pi/4 })$.

\begin{lem}[\cite{ferrari2013tracy}]
\label{lemmasteepdescent1}
\begin{enumerate}
\item[1)] The two following expressions for $f_0'$ are useful:
\begin{eqnarray}
f_0'(Z) &=& \frac{\Psi'_q(\theta)}{\log q}(q^{Z-\theta} -1) + \Psi_q (\theta) -\Psi_q(Z)\label{equationf0primepsi}\\
&=&-\log q \sum_{k=0}^{\infty} \frac{q^{2k}(q^{\theta}-q^Z)^2}{(1-q^{\theta+k})^2(1-q^{Z+k})}.
\label{equationseriesf0prime}
\end{eqnarray}
\item[2)] We have that 
\begin{equation}
 \frac{\mathrm{d}}{\mathrm{d}Y}\left(\Re\left[f_0(X+iY)\right] \right) = 
  - \sin(Y \log q) \log q \sum_{k=0}^{\infty}q^{X+k} \left( \frac{1}{(1-q^{\theta + k)^2}} -\frac{1}{\vert 1 - q^{X+iY +k}\vert^2}\right ) .
\label{equationpartieimaginairef0prime}
\end{equation}
\item[3)]  The contour $\mathcal{C}_{\theta, \pi/4}$ is steep-descent for $-\Re\left[ f_0\right]$, in the sense that the function $s\mapsto  \Re\left[ f_0(W(s))\right]$ is increasing for $s\geqslant 0$, where $W(s)$ is a parametrization of the contour $\mathcal{C}_{\theta, \pi/4}$.
\item[4)]  The function $\Re\left[ f_0\right]$ is periodic on $\lbrace \theta+\sigma +i\R \rbrace$ with period $2\pi/\vert \log q\vert$.  Moreover, $t\mapsto \Re\left[ f_0(\theta+\sigma +it)\right]$ is decreasing on $[0, -\pi/\log q]$ and increasing on  $[ \pi/\log q, 0]$,
 for any $\sigma >0$.
\end{enumerate}
\end{lem}
\begin{proof}
\begin{enumerate}
Equations (\ref{equationseriesf0prime}) and (\ref{equationf0primepsi}) correspond to Equations (6.19) and (6.22) in \cite{ferrari2013tracy}.
Equation (\ref{equationpartieimaginairef0prime}) is Equation (6.24) in \cite{ferrari2013tracy} with $X=\theta+\gamma$ and $ Y=t$.
 4) follows directly from this expression.
 3) is a particular case of of Lemma 6.8 in \cite{ferrari2013tracy} and still holds when $q^{\theta} >1/2$. Indeed, after some algebra,
\begin{equation*}
\frac{\mathrm{d}}{\mathrm{d}s}\left( \Re\left[ f_0(W(s))\right]\right) =  \sum_{k=0}^{\infty} \frac{q^{2k}s^2 \sqrt{2}/2 \left(q^k s^2   + q^{\theta} (1-q^{\theta+k})\right ) }{(1-q^{\theta+k})^2\vert 1-q^{\theta+k}-e^{i\pi/4}sq^k\vert^2 \vert q^{\theta +e^{i\pi/4}s}  \vert^2} >0.
\end{equation*}
\end{enumerate}
\end{proof}
\begin{lem}[\cite{ferrari2013tracy}]
The kernel $K_x(W(s), W')$ has exponential decay, in the sense that there exist $N_0$,  $s_0 \geqslant 0$ and $c>0$ such that for all $s>s_0$ and $N>N_0$, 
\begin{equation*}
\vert K_x(W(s), W') \vert \leqslant \exp(-cNs).
\end{equation*} 
\label{lemmaexponentialdecay}
\end{lem}
\begin{proof}
This is a particular case of Lemma 6.10 in \cite{ferrari2013tracy}, i.e. when $\varphi=\pi/4$. The proof consists in estimating separately the contributions of the vertical line $\theta+\sigma+i\R $ and the small circles $\mathcal{E}_1 , \dots, \mathcal{E}_{k_W}$. The factor $\exp(N(f_0(Z)-f_0(W)))$ inside the kernel commands the asymptotic behaviour. Thus the result boils down to showing that there exists a constant $c>0$ such that for $N>N_0$, $s>s_0$ and any $Z\in\mathcal{D}_W$, 
\begin{equation*}
\Re\left[f_0(Z) - f_0(W(s)) \right] < -c s.
\end{equation*}
This follows  from the properties of the function $f_0$ given in Lemma \ref{lemmasteepdescent1}. Note that this result is also a degeneration of Lemma \ref{lemmaexponentialdecay2} proved thereafter. 
\end{proof}

The previous Lemma allows to extend the contour $\mathcal{C}_{\theta, \varphi} $ with $\varphi\in (0, \pi/4) $  in Theorem \ref{theofredholmdet}   to $\varphi=\pi/4$, without altering the Fredholm determinant $\det(I + K_x)_{\mathbb{L}^2(\mathcal{C}_{\theta, \varphi})} $. Indeed, for $\varphi\in(0, \pi/4)$, it is known from the proof of Theorem 4.13 in \cite{borodin2012free} that the kernel decays exponentially on $\mathcal{C}_{\theta, \varphi}$.
For large $N$, Lemma \ref{lemmaexponentialdecay} gives an exponential bound on the kernel $K_x$ along the tails of the contour $ \mathcal{C}_{\theta, \pi/4}$, i.e. for $\vert s\vert \geqslant  s_0$.  The behaviour of the kernel around $s=0$ is logarithmic, so that for a fixed $N >N_0$ one has (cf also \cite[equation 6.28]{ferrari2013tracy}) 
\begin{equation}
\vert K_x(W(s), W(s'))\vert \leqslant C \exp(-cN\vert s\vert) + C ( \log \vert s\vert )_{-}(\log \vert s'\vert )_-.
\label{equationbornekernelintegrable}
\end{equation}
where $(x)_-$ denotes the negative part of $x$.  Hence, for any fixed $N >N_0$, each term in the series expansion of the Fredholm determinant is constant when $\varphi$ varies in $(0, \pi/4]$, yielding the validity of the contour deformation for the Fredholm determinant.

Next, we want to show that the parts of the contours which give the main contribution to the Fredholm determinant are in a neighbourhood of $\theta$. When $q^{\theta}\leqslant 1/2$ and the contour for the variables $ W$ and $ W'$ is $\mathcal{C}_{\theta, \varphi}$ with $\varphi$ close to $\pi/2$, it is proved in Proposition 6.3 of \cite{ferrari2013tracy}. In order to get rid of the condition $q^{\theta}\leqslant 1/2$, we need to control the real part of $f_0$ on the small circles $\mathcal{E}_1, \dots , \mathcal{E}_{k_W}$. This is done by the following Lemma.

\begin{lem}
There exists $\eta>0$ such that for any $W\in\mathcal{C}_{\theta, \pi/4}$, $\Re(f_0(W)-f_0(W+j))>\eta$ , for all $j=1, \dots, k_W$.
\label{lemmacontributionresidus}
\end{lem}
\begin{proof}
It is proved in Lemma 6.10 in \cite{ferrari2013tracy}  (in the proof thereof, more exactly) that for $W$ far enough  from $\theta$, i.e. for $W=W(s)$ with $\vert s\vert >s_0$, 
\begin{equation*}
\Re\left[f_0(W(s))-f_0(W(s)+j)\right] > c\cdot \vert s\vert
\end{equation*}
for  some $c>0$.  Thus, we can consider the residues lying only in a compact domain, and we are left to prove that $\Re\left[f_0(W)-f_0(W+j)\right]>0$ for each residue.

We split the proof into two cases according to the sign of $\Re\left[q^{W+j}-q^{\theta}\right]$, or in other words, according to the relative position of the residue $W+j$ and the contour $\mathcal{C}_{\theta, \pi/2}$. By symmetry, we can consider only the residues above the real axis.

\item[\textbf{Case 1 :  $\Re\left[q^{W+j}-q^{\theta}\right]\geqslant0$.} ] This condition geometrically means that $W+j$ lies on the left of $ \mathcal{C}_{\theta, \pi/2}$, i.e. between $\mathcal{C}_{\theta, \pi/4}$ and $\mathcal{C}_{\theta, \pi/2}$. We show that on the straight line from $ W$ to $W+j$, $\Re[f_0]$ is decreasing. For that purpose, it is enough to prove that $\Re\left[ f_0'(W+X)\right] <0$ for $X\in(0, j)$. From the expression of $f_0'$ in Lemma \ref{lemmasteepdescent1} eq. (\ref{equationseriesf0prime}),
\begin{equation*}
\Re\left[\frac{\mathrm{d}}{\mathrm{d}X}f_0(W+X)\right] = -\log q \sum_{k=0}^{\infty} \frac{q^{2k}(q^{\theta}-q^{W+X})^2\overline{(1-q^{W+X+k})}}{(1-q^{\theta+k})^2\vert1-q^{W+X+k}\vert^2}.
\end{equation*}
Writing $q^{W+X}=q^{\theta}+ z'$,  the $k$-th term in the series above has the same sign as 
\begin{equation}
\Re\left[(q^{\theta}-q^{W+X})^2\overline{(1-q^{W+X+k})}\right]= (z'^2 \overline{(1-q^k(q^{\theta}+z'))}) = z'^2 (1-q^{\theta} q^{k}) - z' \vert z'\vert ^2 q^{k}.
\label{equationnumeratorf0prime}
\end{equation}
If $\Re\left[q^{W+j}-q^{\theta}\right]\geqslant 0$, then  $\arg (z')\leqslant \pi/2$. Moreover, since $W\in\mathcal{C}_{\theta, \pi/4} $, $W+X$ is on the right of $\mathcal{C}_{\theta, \pi/4}$, which exactly means that   $\arg (z')\geqslant \pi/4$. Hence both terms in the right-hand-side of (\ref{equationnumeratorf0prime}) have negative real part.

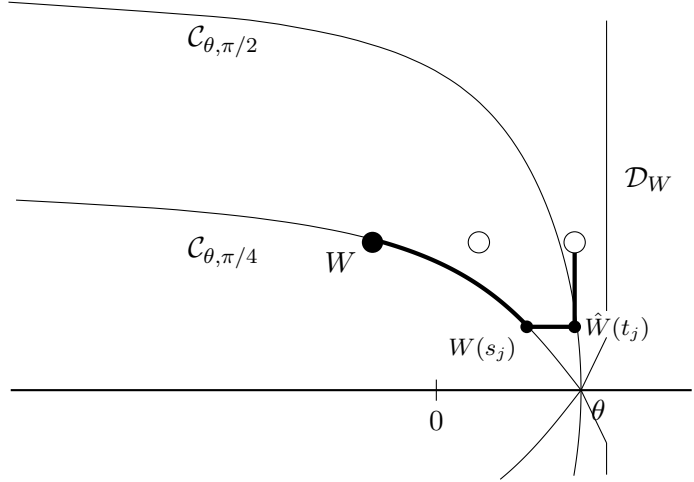
\begin{figure}
\begin{center}
\begin{tikzpicture}[scale=1.4]
\draw[thick] (-4, 0) -- (2.4,0);
\draw (0,0.1) -- (0,-0.1) node[anchor=north]{0};
\draw[scale=1, domain=0:1.385, smooth, variable=\t] plot({1.36 - tan(\t r)},{1.3*\t}); 
\draw[scale=1, domain=0:0.65, smooth, variable=\t] plot({1.36 - tan(\t r)},{-1.3*\t});

\draw[scale=1, domain=0:1.178, smooth, variable=\t] plot({1.36 - tan(\t^2 r)},{3.12*\t});
\draw[scale=1, domain=0:0.26, smooth, variable=\t] plot({1.36 - tan(\t^2 r)},{-3.12*\t});

\fill[] (-0.6, 1.4) circle(0.1);
\draw (0.4, 1.4) circle(0.1);
\draw[] (1.3, 1.4) circle(0.1);
\fill[] (1.3, 0.6) circle(0.06);
\fill[] (0.85, 0.6) circle(0.06);

\draw[ultra thick] (1.3, 1.3) -- (1.3, 0.6);
\draw[ultra thick] (1.32, 0.6) -- (0.85, 0.6);
\draw[ultra thick, scale=1, domain=0.45:1.1, smooth, variable=\t] plot({1.36 - tan(\t r)},{1.3*\t}); 

\draw (1.36,0) -- (1.6,0.5);
\draw (1.36,0) -- (1.6,-0.5);
\draw (1.6, 0.75) -- (1.6, 3.5);
\draw (1.6,- 0.5) -- (1.6, -0.8);

\draw (2,2) node{$\mathcal{D}_W $};
\draw (-0.9, 1.4) node[anchor = north] {$W$};
\draw (0.85, 0.6) node[anchor = north east] {\footnotesize{$W(s_j)$}};
\draw (1.3, 0.6) node[anchor = west] {\footnotesize{$\hat{W}(t_j)$}};
\draw (1.36,0) node[anchor=north west] {$\theta$};
\draw (-2, 1.3) node{$\mathcal{C}_{\theta, \pi/4} $};
\draw (-2, 3.3) node{$\mathcal{C}_{\theta, \pi/2} $};
\end{tikzpicture}
\end{center}
\caption{The thick line is the path from $W$ to $W+j$ in Case $2$, for $j=2$.}
\label{figurecheminresidus}
\end{figure}

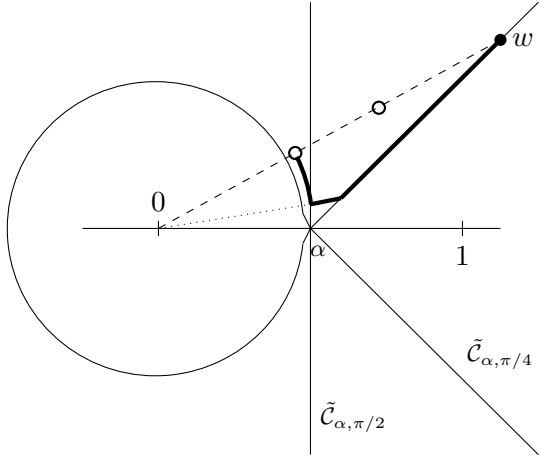
\begin{figure}
\begin{center}
\begin{tikzpicture}[scale=1]
\draw (0, -0.1) -- (0, 0.1) node[above] {$0$};
\draw (4, -0.1)node[below] {$1$} -- (4, 0.1) ;
\draw (-1,0) -- (4.5,0);
\draw (2,-3) -- (2,3) ; 
\draw (2,0) -- (5,3);
\draw (2,0) -- (5, -3);
\draw (1.9,0.2) arc (6:354:1.95);

\draw (1.9, 0.2) -- (2,0);
\draw (1.9, -0.2) -- (2,0);

\draw[dashed] (0,0) -- (4.5,2.5);
\fill (4.5, 2.5) circle (0.08);

\draw [thick] (1.8, 1) circle (0.08);
\draw [thick] (2.9, 1.6) circle (0.08);

\draw[ultra thick] (1.82,0.95) arc(26:7:2);
\draw[dotted] (0,0) -- (2,0.32);
\draw [ultra thick ]  (2, 0.32) -- (2.4, 0.4);
\draw[ultra thick]  (2.4, 0.4) -- (4.5, 2.5) node[right] {$w$};

\draw (2.1,-0.1) node[below]{\footnotesize{$\alpha$}};

\draw (4.5,-2) node[above]{\footnotesize{$\tilde{\mathcal{C}}_{\alpha, \pi/4}$}};
\draw (2,-2.5) node[right]{\footnotesize{$\tilde{\mathcal{C}}_{\alpha, \pi/2}$}};
\end{tikzpicture}
\end{center}
\caption{Image of Figure \ref{figurecheminresidus} by the map $x\mapsto q^x$.}
\label{fig:imageoffigurecheminresidus}
\end{figure}

\item[\textbf{Case 2 :  $\Re\left[q^{W+j}-q^{\theta}\right]<0$. }] This condition geometrically means that $W+j$ lies on the right of  $\mathcal{C}_{\theta, \pi/2}$. Now, it may happen that $\Re[ f_0 ]$ is not decreasing on the horizontal line between $W$ and $W+j$. The idea here, inspired from the proofs of Lemmas 6.10 and 6.12 in \cite{ferrari2013tracy}, is to find another path from $W$ to $W+j$ along which $\Re[f_0]$ is decreasing.

Let $\hat{W}(t) = \log_q(q^{\theta}+e^{i\pi/2 \sgn(t)}\vert t\vert)$ a parametrization of $\mathcal{C}_{\theta, \pi/2}$. 
Let $t_j$ be the real number such that $\Re \left[W+j\right] = \Re\left[\hat{W}(t_j)\right]$ (see Figure \ref{figurecheminresidus} and Figure \ref{fig:imageoffigurecheminresidus}). Let $s_j$ be the real such that $\Im \left[W(s_j)\right] = \Im\left[\hat{W}(t_j)\right]$. 
From $W$ to $W(s_j)$ along the contour $\mathcal{C}_{\theta, \pi/4}$, $\Re[f_0]$ is decreasing by steep-descent property of this contour stated in Lemma \ref{lemmasteepdescent1}.
From $W(s_j)$ to $\hat{W}(t_j) $ on a horizontal line, $\Re [f_0]$ is decreasing from the first part of the proof, because for any $Z$ on this line, we have $\Re\left[q^{Z}-q^{\theta}\right]\geqslant 0$. 
It remains to prove that on the vertical line from $\hat{W}(t_j)$ to $W+j$, $\Re[f_0]$ is decreasing. It is enough to prove that 
\begin{equation*}
\forall Y\in \left(0, \Im[W+j - \hat{W}(t_j)]\right), \  \frac{\mathrm{d}}{\mathrm{d}Y}\left(\Re\left[f_0(\hat{W}(t_j)+iY)\right] \right ) <0.
\end{equation*}
Each summand in the series representation for $\frac{\mathrm{d}}{\mathrm{d}Y}\left(\Re\left[f_0(X+iY)\right]\right)$ in Equation (\ref{equationpartieimaginairef0prime})  has the same sign as $\vert 1-q^{\hat{W}(t_j)+iY+k} \vert^2 - (1-q^{\theta+k})^2 $. This last quantity is positive when $\Re\left[q^{\hat{W}(t_j)+iY}-q^{\theta}\right]<0$. Taking into account the negative prefactor $-\sin(Y\log q)\log q <0 $ in Equation (\ref{equationpartieimaginairef0prime}), we conclude that $\frac{\mathrm{d}}{\mathrm{d}Y}\left(\Re\left[f_0(\hat{W}(t_j)+iY)\right] \right ) <0$.

It may also happen that  $\Re \left[W+j\right] = \theta+ \sigma'$ with $0<\sigma'< 2\sigma$, and in this case the path we have just described does not exist. But it suffices to go from $W$ to $\theta$ along $\mathcal{C}_{\theta, \pi/4}$, from $\theta$ to $\theta+ \sigma'$ by a straight horizontal line, and finally to $W+j$ by a vertical line. On the short horizontal line, $\Re [f_0]$ may increase, but the derivative is bounded, and $\sigma$ can be chosen as small as we want. Hence the possible increase of $\Re\left[f_0 \right] $ from $\theta$ to $\theta+2\sigma$ can be made arbitrarily small, which is enough to prove the Lemma.
\end{proof}
We are now able to prove the following analogue of \cite[Proposition 6.3]{ferrari2013tracy}.
\begin{prop} Asymptotically, the contribution to the Fredholm determinant of the parts of the contours outside any neighbourhood of $\theta$ is negligible. More rigorously, 
for any fixed $\delta >0$ and $\epsilon >0$, there is an $N_1$ such that for all $N>N_1$
\begin{equation*}
\vert \det(I+K_x)_{\mathbb{L}^2(\mathcal{C}_{\theta, \pi/4})} - \det(I+K_{x, \delta})_{\mathbb{L}^2(\mathcal{C}_{\theta}^{\delta})}\vert < \epsilon
\end{equation*}
 where $\mathcal{C}_{\theta}^{\delta}$ is the truncated contour $ \mathcal{C}_{\theta, \pi/4} \cap \lbrace W\ ; \ \vert W-\theta\vert \leqslant \delta \rbrace $, and 
\begin{multline}
K_{x, \delta}(W,W')= \\ 
\frac{q^W \log q}{2i\pi} \int_{\mathcal{D}_W^{\delta}} \frac{\mathrm{d}Z}{q^Z - q^{W'}}\frac{\pi}{\sin(\pi(Z-W))} \frac{\exp(Nf_0(Z)+N^{2/3} f_1(Z)+N^{1/3}f_2(Z))}{\exp(Nf_0(W)+N^{2/3} f_1(W)+N^{1/3}f_2(W))}\frac{\phi(Z)}{\phi(W)}
\label{equationmainexponentialterm}
\end{multline}
and analogously  $\mathcal{D}_{W}^{\delta} = \mathcal{D}_{W} \cap \lbrace Z\ ; \ \vert Z-\bar{A}\vert \leqslant \delta \rbrace $.
\label{propositionlocalization}
\end{prop}
\begin{proof}
This Proposition is the precise adaptation of Proposition 6.3 in \cite{ferrari2013tracy} and we reproduce the proof done therein. We have the Fredholm determinant expansion
\begin{equation}
\det(I + K_x)_{\mathbb{L}^2(\mathcal{C}_{\theta, \pi/4})} = \sum_{k=0}^{\infty} \frac{1}{k!} \int_{\R}\mathrm{d}s_1 \dots \int_{\R}\mathrm{d}s_k \det\left(K_x\left(W(s_i), W(s_j) \right)_{1\leqslant i, j \leqslant k}\right) \frac{\mathrm{d}W(s_i)}{\mathrm{d}s_i} .
\label{equationfredholmexpansion}
\end{equation}
Let us denote by $s_{\delta}$ the positive real number such that $\vert W(s_{\delta}) -\theta\vert =\delta $. We need to prove that if we replace all the integrations on $\R$ in (\ref{equationfredholmexpansion}) by integrations on $[-s_{\delta}, s_{\delta}]$, the error that we make goes to zero when $N$ goes to infinity. We give a dominated convergence argument. Note that the integrable bound in equation (\ref{equationbornekernelintegrable}) is not useful here since this bound is valid for a fixed $N$.

By Lemma \ref{lemmaexponentialdecay} and Lemma \ref{lemmacontributionresidus} together with the steep-descent properties of the contours, one can find a constant $c_{\delta}>0$ such that for any $N>N_0$ and $\vert s \vert > s_{\delta}$, 
\begin{equation*}
\Re\left[f_0(Z) - f_0(W(s)) \right]<-c_{\delta} s.
\end{equation*}
Furthermore the integral in (\ref{equationmainexponentialterm}) is absolutely integrable. For the vertical part of the contour $\mathcal{D}_W $, this is due to the exponential decay of the sine in the denominator. Thus, one can find another positive constant $ C_{\delta}$ such that for $\vert s \vert > s_{\delta}$  and $N>N_0$, one has
\begin{equation}
\vert K_x(W(s), W')\vert < C_{\delta} \exp\left(-\frac{c_{\delta}}{2}N s\right). 
\end{equation}
Hence, when  $N \geqslant N_0$ the series expansion of the error term, that is the expression in (\ref{equationfredholmexpansion}) with integrations on $\R^k\setminus [-s_{\delta}, s_{\delta}]^k $, can be uniformly bounded by a convergent series of absolutely convergent integrals.
Thus, by dominated convergence, the error goes to zero.

To conclude the proof of the Proposition, we have to localize the $Z$ integrals on $\mathcal{D}_W^{\delta} $, and we outline the arguments of \cite{ferrari2013tracy}. The behaviour in the $Z$ variables is $e^{-\pi \Im[Z]} $ due to the sine in the denominator. Hence by the steep-descent property of the contour for $Z$ on each period,  and the same kind of dominated convergence arguments, one can localize the $Z$ integrals in neighbourhoods of size $\delta$ around each $\theta + i2k\pi/\log q $ for $k\in \Z$. Moreover one can show that the contribution of the integrals on these $\delta$-neighbourhoods is $\mathcal{O}(N^{-1/3})$ as soon as $k\neq 0$, and summable on $k$.  
\end{proof}

We make the change of variables
\begin{equation*}
W=\theta + wN^{-1/3}, W'=\theta + w'N^{-1/3}, Z=\theta + zN^{-1/3}.
\end{equation*}
In order to adapt the rest of the asymptotic analysis made in \cite{ferrari2013tracy}, we need the following estimate on the behaviour of our additional factor inside the kernel.
\begin{lem}
For any $z$ and $w$, we have
\begin{equation}
\frac{\phi(\theta + zN^{-1/3})}{\phi(\theta + wN^{-1/3})} \underset{N\to\infty}{\longrightarrow} 1.
\label{equationlimitphi}
\end{equation}
Moreover, there exist constants $c_{\phi}, C_{\phi}>0$ such that for $\vert Z-\theta \vert < c_{\phi}$ and $\vert W-\theta \vert < c_{\phi} $, one has 
\begin{equation}
\left| \frac{\phi(Z)}{\phi(W)}\right| \leqslant C_{\phi}.
\label{equationboundphi}
\end{equation}
\label{lemmapointwiselimit1}
\end{lem}
\begin{proof}
The infinite product $(z ; q)_{\infty}$ converges uniformly on any disk centred in 0. Here for all $1\leqslant j \leqslant m$, $q^{\theta}<\alpha\leqslant a_{i_j}$. Thus, each factor tends to a non null real number, and one can  exchange limit and infinite product. The limit does neither depend on $z$ nor on $w$, and one has
\begin{equation}
\frac{(q^{\theta + zN^{-1/3}}/a_{i_j} ; q)_{\infty}}{\left((q^{\theta + zN^{-1/3}} ; q)_{\infty}\right)} \underset{N\to\infty}{\longrightarrow} \frac{(q^{\theta}/a_{i_j} ; q)_{\infty}}{\left((q^{\theta} ; q)_{\infty}\right)}.
\label{equationlimitperturbation}
\end{equation}
The factors in $\phi(\theta + zN^{-1/3})$ and $\phi(\theta + wN^{-1/3})$ compensate in the limit,  and $\frac{\phi(\theta + zN^{-1/3})}{\phi(\theta + wN^{-1/3})} \underset{N\to\infty}{\longrightarrow} 1$.
  
 Assuming $\vert Z-\theta \vert < c_{\phi}$ and $\vert W-\theta \vert < c_{\phi} $  where  $c_{\phi} $ is chosen small enough,   $\vert q^{Z+k}/a_{i_j}\vert$ is uniformly bounded by a constant smaller than $1$. Hence $\phi(Z)$ and $\phi(W)$ are uniformly bounded above and below by positive constants, and one can find a constant $C_{\phi}$ so that (\ref{equationboundphi}) holds.
\end{proof}

Due to the change of variables, we define new integration contours which we choose as straight lines for simplicity. 
For $L\in\R_+\cup\lbrace\infty \rbrace$, the contours $\mathcal{C}_{\varphi, L}$ and $\mathcal{D}_{\varphi, L}$ are  adapted from  \cite{ferrari2013tracy} and defined in the following way:  $\mathcal{C}_{\varphi, L} = \lbrace  e^{i (\pi-\varphi)\sgn(y)}\vert y \vert, \vert y \vert\leqslant L\rbrace $ for some angle $\varphi<\pi/4$. 
Analogously we define $\mathcal{D}_{\varphi, L}=\lbrace  e^{i\varphi \sgn(y)}\vert y\vert , \vert y \vert\leqslant L \rbrace $.  This modification of contours can be performed without changing the value of the integral as soon as we keep the same endpoints, and the angle $\varphi$ and the parameter $\sigma$ can be chosen so that it is the case.

\begin{prop} We have the convergence
\begin{equation*}
\lim_{N\to\infty}\det(I+K_x)_{\mathbb{L}^2(\mathcal{C}_{\theta, \pi/4})} =\det(I+K'_{x, \infty})_{\mathbb{L}^2(\mathcal{C}_{\varphi, \infty})}
\end{equation*}
 where for $L\in\R_+\cup \lbrace +\infty\rbrace$,  
\begin{equation}
K'_{x, L} = \frac{1}{2i\pi} \int_{\mathcal{D}_{\varphi, L}} \frac{\mathrm{d}z}{(z-w')(w-z)}\frac{\exp(\chi z^3/3 + c (\log q )^2 z^2/2 + z c^2 (\log q)^4/(4\chi)  -z x\chi^{1/3} )}{\exp(\chi w^3/3 + c (\log q )^2 w^2/2 + w c^2 (\log q)^4/(4\chi) - w x\chi^{1/3} )}.
\label{equationmainexponentialtermlimit}
\end{equation}

\label{propositionlimitfredholm}
\end{prop} 
\begin{proof}
For the sake of self-containedness, we reproduce the proofs of Propositions 6.4 to 6.6 of \cite{ferrari2013tracy} which still hold with a slight modification for the pointwise limit. 

Consider the rescaled kernel 
\begin{equation*}
K_{x, \delta}^N(w,w') = N^{-1/3} K_{x, \delta N^{1/3}}(\theta + w N^{-1/3}, \theta + w' N^{-1/3})
\end{equation*}
where we use the contours $\mathcal{C}_{\varphi, \delta N^{1/3}}$ and $\mathcal{D}_{\varphi, \delta N^{1/3}}$.
By a simple change of variables, 
\begin{equation*}
\det(I+K_{x, \delta})_{\mathbb{L}^2(\mathcal{C}_{\theta}^{\delta})} =  \det(I+K_{x, \delta}^N)_{\mathbb{L}^2(\mathcal{C}_{\varphi, \delta N^{1/2}})}.
\end{equation*}
First we estimate the argument in the exponential  in (\ref{equationmainexponentialterm}). By Taylor approximation, there exists $C_{f_0}$, such that for $\vert Z-\theta\vert <\theta $,
\begin{equation}
\left| f_0(Z) - f_0(\theta) - \frac{\chi}{3}(Z-\theta)^3 \right| < C_{f_0} \vert Z-\theta\vert^4
\label{equationtaylorexpansionf0}
\end{equation}
and since $f_1'(\theta)=0$ and $f_1''(\theta) = c(\log q)^2 $, there exists $C_{f_1}$, such that for $\vert Z-\theta\vert <\theta $,
\begin{equation*}
\left| f_1(Z) - f_1(\theta) - c(\log q)^2 (Z-\theta)^2 \right|< C_{f_1}\vert Z-\theta \vert^3.
\end{equation*}
Let us denote the argument in the exponential  in (\ref{equationmainexponentialterm}) as 
\begin{equation*}
f(Z,W,N) := Nf_0(Z) + N^{2/3}f_1(Z) + N^{1/3} f_2(Z) - Nf_0(W) - N^{2/3}f_1(W)-  N^{1/3} f_2(W), 
\end{equation*}
and the argument in the exponential in  (\ref{equationmainexponentialtermlimit}) as 
\begin{multline*}
f^{\mathrm{lim}}(z,w):= \left(\chi z^3/3 + c (\log q )^2 z^2/2 + c^2 (\log q)^4/(4\chi) z -x\chi^{1/3} z\right) \\-\left( \chi w^3/3 + c (\log q )^2 w^2/2 + c^2 (\log q)^4/(4\chi) w -x\chi^{1/3} w\right).
\end{multline*}
Using the Taylor approximations above and rescaling the variables, one has that for $w\in \mathcal{C}_{\varphi, \delta N^{1/3}}$,  $z\in\mathcal{D}_{\varphi, \delta N^{1/3}}$, and $Z=\theta + zN^{-1/3}, W=\theta + wN^{-1/3} $, 
\begin{eqnarray}
\left|f(Z,W,N) - f^{\mathrm{lim}}(z,w)\right| &<&N^{-1/3}\left(C_{f_0} (\vert z\vert^4+\vert w\vert^4)+C_{f_1} (\vert z\vert^3+\vert w\vert^3) \right) \label{equationestimeeexponentielle1}\\
& \leqslant & \delta \left(C_{f_0} (\vert z\vert^3+\vert w\vert^3)+C_{f_1} (\vert z\vert^2+\vert w\vert^2) \right).\label{equationestimeeexponentielle2}
\end{eqnarray}
Now we estimate the remaining factors in the integrand in (\ref{equationmainexponentialterm}). Let us denote
\begin{equation*}
F(Z,W,W'):= \frac{N^{-1/3}}{q^Z-q^{W'}} \frac{\pi}{\sin(\pi(Z-W))}\frac{\phi(Z)}{\phi(W)}, 
\end{equation*}
and the remaining factors in the integrand in (\ref{equationmainexponentialtermlimit}) as 
\begin{equation*}
F^{\mathrm{lim}}(z,w,w'):= \frac{1}{z-w'}\frac{1}{z-w}.
\end{equation*}
Let us prove that for any $w,w' \in \mathcal{C}_N$, $K_{x, \delta}^N(w,w') - K'_{x,\delta N^{1/3}}(w,w') $ goes to zero when $N$ goes to infinity. Indeed, the error can be estimated by 
\begin{multline}
\vert K_{x, \delta}^N(w,w') - K'_{x,\delta N^{1/3}}(w,w') \vert < \int_{\mathcal{D}_N}\mathrm{d}z \exp(f^{\mathrm{lim}}) \vert F(Z, W, W' )\vert  \left| \exp(f - f^{\mathrm{lim}})-1\right| \\
+ \int_{\mathcal{D}_N}\mathrm{d}z \exp(f^{\mathrm{lim}}) \left|F - F^{\mathrm{lim}} \right|, 
\label{equationestimeeerreur}
\end{multline}
where we have omitted the arguments of the functions $f(Z, W, N)$, $f^{\mathrm{lim}}(z,w)$, $F(Z,W,W')$, $F^{\mathrm{lim}}(z,w,w')$, with $Z=\theta + zN^{-1/3}$ as before, and likewise for $W, W'$. By estimates  (\ref{equationestimeeexponentielle1}) and (\ref{equationestimeeexponentielle2}) and the inequality $\vert \exp(x) - 1\vert \leqslant \vert x\vert \exp(\vert x\vert)$, we have 
\begin{equation*}
\vert \exp(f - f^{\mathrm{lim}})-1\vert  < N^{-1/3} P(\vert z\vert, \vert w \vert) \exp\left(\delta \left(C_{f_0} (\vert z\vert^3+\vert w\vert^3)+C_{f_1} (\vert z\vert^2+\vert w\vert^2) \right) \right), 
\end{equation*}
where $P$ is the polynomial $P(X,Y) = C_{f_0} (X^4+Y^4)+C_{f_1} (X^3+Y^3)$. Hence, for $\delta$ small enough, 
$$\exp(f^{\mathrm{lim}}) \vert \exp(f - f^{\mathrm{lim}})-1\vert $$
has cubic exponential decay in $\vert z\vert$ when $z$ goes to infinity along the contour $\mathcal{D}_{\infty}$. Hence  
the first integral in (\ref{equationestimeeerreur}) goes to zero as $N$ goes to infinity by dominated convergence. The second integral in   (\ref{equationestimeeerreur}) also goes to zero by dominated convergence since one can bound
\begin{multline*}
\left| F(\theta +zN^{-1/3}, \theta +wN^{-1/3}, \theta +w'N^{-1/3})- F^{\mathrm{lim}}(z,w,w')\right| < \\ N^{-1/3} Q(\vert z \vert, \vert w \vert, \vert w' \vert) F^{\mathrm{lim}}(z,w,w'), 
\end{multline*}
for some polynomial $Q$.

In order to prove that the difference of Fredholm determinants goes to zero as well, one could show that the difference of operators $K_{x, \delta}^N$ and  $K'_{x,\delta N^{1/3}}$ acting on $\mathbb{L}^2(\mathcal{C}_{\infty}) $ goes to zero in trace-class norm, but we give a simpler dominated convergence argument instead.
The estimates in right-hand-sides of Equations (\ref{equationestimeeexponentielle2}) and (\ref{equationboundphi}) show that $K_{x, \delta}^N$ has cubic exponential decay. More precisely, there exists a constant $C>0$ independent of $N$ such that for all $w, w'\in\mathcal{C}_{\varphi, \delta N^{1/3}}$,
 \begin{equation*}
\vert K_{x, \delta}^N (w,w') \vert < C  \exp\left(f^{\mathrm{lim}}(0,w)+C_{f_0}\delta \vert w \vert^3  + C_{f_1}\delta \vert w\vert^2\right).
 \end{equation*}
 Hence for $\delta$ small enough, Hadamard's bound yields
\begin{equation*}
\left|\det\left(K_{x, \delta}^N (w_i,w_j)_{1\leqslant i,j\leqslant n} \right ) \right| \leqslant n^{n/2} C^n \prod_{i=1}^n e^{\chi/6\Re \left[w_i^3\right]}.
\end{equation*}
It follows that  the Fredholm determinant expansion,
\begin{equation*}
\det(I+K_{x, \delta}^N)_{\mathbb{L}^2(\mathcal{C}_{\varphi, \delta N^{1/3}})} = \sum_{n=0}^{\infty} \frac{1}{n!} \int_{\mathcal{C}_{\varphi, \delta N^{1/3}}}\mathrm{d}w_1 \dots \int_{\mathcal{C}_{\varphi, \delta N^{1/3}}}\mathrm{d}w_n \det\left(K_{x, \delta}^N (w_i,w_j)_{1\leqslant i,j\leqslant n} \right ),
\end{equation*}
is absolutely integrable and summable. Thus, by dominated convergence, 
\begin{equation*}
\lim_{N\to\infty}\det(I+K_x)_{\mathbb{L}^2(\mathcal{C}_{\theta, \pi/4})} =\lim_{N\to\infty}\det(I+K'_{x,\delta N^{1/3}})_{\mathbb{L}^2(\mathcal{C}_{\varphi, \delta N^{1/3}})}.
\end{equation*}

Since the integrand in $K'_{x,\delta N^{1/3}}$ has cubic exponential decay along the contours $\mathcal{C}_{\infty}$ and  $\mathcal{D}_{\infty}$, dominated convergence, again,  yields 
\begin{equation*}
\det(I+ K_{x})_{\mathbb{L}^2(\mathcal{C}_{\theta, \pi/4 })} \underset{N\to\infty}{\longrightarrow} \det(I+ K'_{x,\infty})_{\mathbb{L}^2(\mathcal{C}_\infty)}.
\end{equation*}
\end{proof}

Now we explain how the limit of the $q$-Laplace transform characterizes the limit law of the rescaled position of particles. The sequence of functions  $E_N(y):= 1/(-q^{-y N^{1/3}} ; q)_{\infty}$ is such that for any $N >0$, $E_N(y)$ is strictly decreasing with limit $1$ when $y$ goes to $-\infty$, and with limit $0$ when $y$ goes to $+\infty$. Additionally, for each $\varepsilon >0$, $E_N$ converges uniformly to $\mathds{1}_{y\leqslant 0}$ on $\R\setminus [-\varepsilon, \varepsilon]$. Using Lemma 4.39 in \cite{borodin2011macdonald} to replace $E_N$ by its limit and with our choice of $\zeta$, 
\begin{eqnarray*}
\lim_{N\to\infty }\mathbb{P}(\xi_N <x) &=& \lim_{N\to\infty}\EE\left[E_N\left(\frac{\chi^{1/3}}{\vert\log q\vert} (\xi_N-x)\right) \right]\\
 &=& \lim_{N\to\infty} \EE\left[\frac{1}{(\zeta q^{X_N(t)+N} ;  q)_{\infty}}\right]\\
 & =& \det(I+K'_{x, \infty})_{\mathbb{L}^2(\mathcal{C}_{\varphi, \infty})}. 
\end{eqnarray*}

Finally, using a classical reformulation of the kernel (see \cite[Lemma 8.7]{borodin2012free} ) to get the Fredholm determinant of an operator acting on $ \mathbb{L}^2(\R_+)$, and after the change of variables $z \leftarrow \chi^{1/3} ( z +c(\log q)^2 /(2\chi) )$ and likewise for $w$ and $w'$, 
\begin{equation*}
\det(I+K'_{x, \infty})_{\mathbb{L}^2(\mathcal{C}_{\varphi, \infty})} = \det(I-K_{\mathrm{Ai}})_{\mathbb{L}^2(x, +\infty)}
\end{equation*}
and  we conclude that
\begin{equation*}
\lim_{N\to\infty}\mathbb{P}(\xi_N <x)=F_{\rm GUE}(x).
\end{equation*}

\subsection{Case $\alpha = q^{\theta}$, critical value.}
\label{sectioncriticalcase}
The function $\phi$ introduces a pole of order $k$ in $A=\theta$ in the kernel $K_x$, for the variable $W$. The contour  of $W$ must enclose this pole, and thus $\mathcal{C}_{\bar{A}, \varphi}$ has to pass on the right of $\theta$.
The contour can be chosen as in the previous section, except for a modification (e.g. a small circle of radius $(\epsilon/2) N^{-1/3}$ centred at $\theta$) in a $N^{-1/3}$-neighbourhood of $\theta$. In order to stay on the right of $\mathcal{C}_{\bar{A}, \varphi} $, the contour $\mathcal{D}_W$ can be simply shifted to the right by $\epsilon N^{-1/3}$. In order to adapt the arguments of the case $\alpha > q^{\theta} $, 
we only need  the pointwise limit and a uniform bound in a neighbourhood of $\theta$
for the factor $\phi(Z)/\phi(W)$ introduced in the kernel.
\begin{lem}
For any $z$ and $w$, we have
\begin{equation*}
\frac{\phi(\theta + zN^{-1/3})}{\phi(\theta + wN^{-1/3})} 
 \underset{N\to\infty}{\longrightarrow} \left(\frac{z}{w}\right)^k.
\end{equation*}
Moreover, there exist constants $c_{\phi}', C_{\phi}'>0$  such that for  $\vert Z-\theta\vert < c_{\phi}' $ and $\vert W-\theta\vert < c_{\phi}' $, one has 
\begin{equation}
\left| \frac{\phi(Z)}{\phi(W)} \right| < C_{\phi}' \left|\frac{Z-\theta}{W-\theta}\right|^k.
\label{equationboundphicritical}
\end{equation}
\label{lemmapointwiselimitcritical}
\end{lem}
\begin{proof}
For $j$ such that $a_{i_j}>q^{\theta}$, the limit in (\ref{equationlimitperturbation}) still holds. We are left with the $k$ factors for which $a_{i_j}=\alpha$. In this case
\begin{equation}
\frac{(q^{\theta + zN^{-1/3}}/\alpha ; q)_{\infty}}{(q^{\theta + zN^{-1/3}} ; q)_{\infty}} = 
\frac{(q^{zN^{-1/3}} ; q)_{\infty}}{(q^{\theta+ zN^{-1/3}} ; q)_{\infty}} 
\underset{N\to\infty}{\sim}  \frac{(-\log q) zN^{-1/3}(q;q)_{\infty}}{(q^{\theta}; q)_{\infty}}\label{equationestimatephicritical}. 
\end{equation}
The $N^{-1/3}$ and constant factors in $\phi(Z)$ and $\phi(W)$ compensate in the limit, and we get the result.

Let us prove the bound (\ref{equationboundphicritical}). For $j$ such that $a_{i_j}>q^{\theta}$ the factors $\frac{(q^{\theta + zN^{-1/3}}/a_{i_j}; q)_{\infty}}{(q^{\theta + zN^{-1/3}} ; q)_{\infty}}$ are bounded as in Lemma \ref{lemmapointwiselimit1}. For the factors for which $a_{i_j}=\alpha$, we use the fact that the function $u\mapsto \vert (1-q^u)/u \vert $ is bounded above and below by positive constants on some disc centred in $0$ of positive radius $r$. Choosing $c_{\phi}' \leqslant r$ and small enough so that Lemma \ref{lemmapointwiselimit1} applies, one gets the result.
\end{proof}
With this Lemma, the local modification of the paths has no influence on any of the bounds given previously for large $w$ and $z$.
Only the pointwise limit of the modified kernel is slightly different and given by the above Lemma. 
We conclude that 
\begin{equation*}
\lim_{N\to\infty} \EE\left[\frac{1}{(\zeta q^{X_N(t)+N} ; q)_{\infty}}\right] = \det(I+K'_x)_{\mathbb{L}^2(\mathcal{C}_{\varphi, \infty})}
\end{equation*}
 where 
\begin{multline}
K'_x =\frac{1}{2i\pi} \int_{\mathcal{D}_{\varphi, \infty}} \frac{\mathrm{d}z}{(z-w')(w-z)} \\ \frac{\exp(\chi z^3/3 + c (\log q )^2 z^2/2 + c^2 (\log q)^4/(4\chi) z -x\chi^{1/3} z)}{\exp(\chi w^3/3 + c (\log q )^2 w^2/2 + c^2 (\log q)^4/(4\chi) w-x\chi^{1/3} w)} \left(\frac{z}{w}\right)^k. 
\label{equationpremierkernelcascritique}
\end{multline}
The contours $\mathcal{C}_{\varphi, \infty}$ and $\mathcal{D}_{\varphi, \infty}$ are slight modifications of those defined in the previous section. Here, $\mathcal{C}_{\varphi, \infty} = \lbrace \theta + e^{i (\pi-\varphi)\sgn(y)}\vert y \vert\ ; \vert y\vert >N^{-1/3}\epsilon/2 \rbrace \cup \lbrace \epsilon/2 N^{-1/3} e^{i\gamma}\ ; \ \gamma\in[\varphi - \pi ; \pi-\varphi]\rbrace$. The contour $\mathcal{D}_{\phi, \infty}$ can be chosen as $\lbrace\epsilon N^{-1/3} + e^{i\varphi \sgn(y)}\vert y\vert, \ y\in\R \rbrace $. 

We reformulate the kernel as a Fredholm determinant acting on $ \mathbb{L}^2(\R_+)$ (see \cite[Lemma 8.7]{borodin2012free}), and after the change of variables $z \leftarrow \chi^{1/3} ( z +c(\log q)^2 /(2\chi) )$ and likewise for $w$ and $w'$, we conclude that
\begin{equation*}
\lim_{N\to\infty}\mathbb{P}(\xi_N <x) = \det(I-K''_x(w,w'))_{\mathbb{L}^2(x, +\infty)}
\end{equation*}
where 
\begin{equation*}
K''_x(u,v) = 
\frac{1}{(2i\pi)^2} \int_{e^{-2i\pi/3}\infty}^{e^{2i\pi/3}\infty} \mathrm{d}w \int_{e^{-i\pi/3}\infty}^{e^{i\pi/3}\infty} \mathrm{d}z \frac{e^{z^3/3-zu}}{e^{w^3/3-wv}}\frac{1}{z-w}\left(\frac{z-\frac{c(\log q)^2}{2\chi^{2/3}}}{w-\frac{c(\log q)^2}{2\chi^{2/3}}}\right)^k,
\end{equation*}
where the contour for $w$ passes to the right of $b:=\frac{c(\log q)^2}{2\chi^{2/3}}$, and the contours for $z$ and $w$ do not intersect. Finally, by Definition \ref{defdistributions}, 
\begin{equation*}
\lim_{N\to\infty}\mathbb{P}(\xi_N <x)=F_{\rm BBP,k, \mathbf{b}}(x),
\end{equation*}
with $\mathbf{b} = (b, \dots, b)$.
\begin{rem}
In the case where for $1\leqslant i\leqslant k$, $a_i = q^{ \theta + \tilde{b}_i N^{-1/3}}$ and the rates of all other particles are higher than $q^{\theta}$, Lemma \ref{lemmapointwiselimitcritical} still applies and  the factor $\left(z/w\right)^k$ in Equation (\ref{equationpremierkernelcascritique}) has to be replaced by $\prod_{i=1}^{k}(z-b_i)/(w-b_i)$. Then 
$\left((z-b)/(w-b) \right)^k$ with $b=c(\log q)^2\chi^{-2/3}/2$ gets replaced by $\prod_{i=1}^{k}(z-b_i)/(w-b_i)$ with $b_i = b+\tilde{b}_i$, and finally
\begin{equation*}
\lim_{N\to\infty}\mathbb{P}(\xi_N <x)=F_{\rm BBP,k, \mathbf{b}}(x), 
\end{equation*}
with $\mathbf{b} = (b_1, \dots, b_k)$.
\end{rem}

\subsection{Case $\alpha < q^{\theta}$, Gaussian fluctuations}

We start again from the result of Theorem \ref{theofredholmdet}.  One cannot use the same contour for the Fredholm determinant, because the pole for $W=A$ in $K_x(W,W') $ has to be inside the contour $\mathcal{C}_{\bar{A}, \varphi}$, which means $\bar{A}\geqslant A >\theta$. 
Let us choose 
\begin{equation*}
\zeta = -q^{-gN -cN^{1/2}- \sigma^{1/2}\frac{N^{1/2}}{\log q}}
\end{equation*}
so that 
\begin{equation*}
\lim_{N\to\infty}  \mathbb{P}\left(\frac{X_N(\tau^*(N,c)) - p^*(N,c)}{N^{1/2} \sigma^{1/2}/(\log q)  } <x \right) =  \lim_{N\to\infty} \mathbb{E}\left[1/(\zeta q^{X_N(\tau)} ; q)_{\infty} \right]
\end{equation*}
with the new macroscopic position $p^*(N,c) = (g-1)N + cN^{1/2}$.

Again, $\det (I+\tilde{K}_{\zeta})_{\mathbb{L}^2(\tilde{\mathcal{C}}_{ \bar{\alpha},\varphi})} = \det(I+K_x)_{\mathbb{L}^2(\mathcal{C}_{ \bar{A},\varphi})}$ where 
\begin{equation}
K_x(w,w')= \frac{q^W \log q}{2i\pi} \int_{\mathcal{D}_W} \frac{\mathrm{d}Z}{q^Z - q^{W'}}\frac{\pi}{\sin(\pi(Z-W))} \frac{\exp(Ng_0(Z)+N^{1/2}g_1(Z))}{\exp(Ng_0(W)+N^{1/2}g_1(W))}\frac{\phi(Z)}{\phi(W)}
\label{equationexpressionkernelgaussien}
\end{equation}
with
\begin{eqnarray*}
g_0(Z) &=& - g \log(q) Z + \kappa q^Z + \log(q^Z; q)_{\infty},\\
g_1(Z) &=& -Z \log(q) c-\sigma^{1/2} x Z + \frac{c}{\alpha} q^Z.
\end{eqnarray*}
The asymptotic behaviour  is governed by the real part of the function $g_0$. By a direct calculation and Equations (\ref{defg}) and (\ref{defkappa}), one has that 
\begin{eqnarray*}
g_0'(Z)&=& -g\log(q) + \log(q) \kappa q^Z + \sum_{k=0}^{\infty} \frac{-\log(q) q^{Z+k} }{1-q^{Z+k}}\\
&=& \frac{\Psi'_q(\theta)}{q^{\theta}\log(q)}(q^Z-\alpha)+\Psi_q(A)-\Psi_q(Z), \\
g_0''(Z) &=& \Psi'_q(\theta)q^{Z-\theta} - \Psi'_q(Z).
\end{eqnarray*} 
We see immediately that $g'_0(A)=0$, and using the series representation (\ref{equationseriespsi}) and (\ref{equationseriespsiprime}), for $A>\theta$,
\begin{equation}
 g_0''(A) = \sigma = (\log q)^2 \sum_{k=0}^{\infty}q^{A+k}\left(\frac{1}{(1-q^{\theta+k})^2} - \frac{1}{(1-q^{A+k})^2}\right) >0.
 \label{equationderiveesecondeg0}
\end{equation}

\begin{lem}
\label{lemmasteepdescent2}
\begin{enumerate}
\item  The contour $\mathcal{C}_{A, \pi/4}$ is steep-descent for $-\Re[ g_0]$ in the sense that the function $s\mapsto \Re[ g_0(W(s))]$ is increasing for $s\geqslant 0$, where $W(s)$ is a parametrization of $\mathcal{C}_{A, \pi/4}$.
\item The function $\Re \left[g_0\right]$ is periodic on $\lbrace A+\sigma +i\R \rbrace$ with period $2\pi/\vert \log q\vert$.  Moreover, $t\mapsto \Re\left[ g_0(A+\sigma +it)\right]$ is decreasing on $[0, -\pi/\log q]$ and increasing on  $[ \pi/\log q, 0]$, 
for any $\sigma >0$.
\end{enumerate}
\end{lem}
\begin{proof}
\begin{enumerate}
\item We assume that $\alpha<q^{\theta}$. Using the parametrization of the contour $\mathcal{C}_{A, \varphi} $ $ W(s)  = \log_q(\alpha + s e^{i\varphi})$ as before , we have
\begin{multline*}
\frac{\mathrm{d}}{\mathrm{d}s} \left( \Re\left[g_0(W(s))\right]\right) =  \sum_{k=0}^{\infty}\left(\frac{s q^k}{(1-q^{\theta+k})^2}\frac{\alpha\cos(2\varphi)+s\cos(\varphi)}{\vert \alpha+se^{i\varphi}\vert^2} + \right. \\
\left. \frac{\alpha q^k}{(1-\alpha q^k )}\frac{\alpha\cos(\varphi)+s}{\vert \alpha+se^{i\varphi}\vert^2}  -\frac{(\cos(\varphi)-(\alpha \cos(\varphi)+s)q^k)q^k}{\vert 1-(\alpha+se^{i\varphi})q^k \vert^2}\right).
\label{steepdescent1}
\end{multline*}
For $\varphi\leqslant \pi/4$, using the fact that $q^{\theta}>\alpha$, and factoring the summand, we get
\begin{multline*}
\frac{\mathrm{d}}{\mathrm{d}s} \left( \Re\left[g_0(W(s))\right]\right) > \\
 \sum_{k=0}^{\infty} \frac{q^{2k}s^2\left(q^k s^2\cos(\varphi) - (1-2\alpha q^k)s\cos(2\varphi) - \alpha(1-\alpha q^k)\cos(3\varphi)\right)}{(1-\alpha q^{k})^2\vert \alpha+se^{i\varphi}\vert^2 \vert 1-(\alpha+se^{i\varphi})q^k \vert^2}, 
\end{multline*}
which is positive for $s>0$ and $\varphi=\pi/4$.
\item Let $Z(t) = A +\sigma +it$. Notice that $g_0'(Z) = -\log(q) (g-f) + f_0'(Z)$, and $\frac{\mathrm{d}}{\mathrm{d}t} \left( \Re\left[g_0(Z(t))\right]\right) = - \Im\left[ g_0'(Z(t))\right]$. By Lemma \ref{lemmasteepdescent1}, 
\begin{equation*}
\frac{\mathrm{d}}{\mathrm{d}t} \left( \Re\left[g_0(Z(t))\right]\right) = -\sin(t \log q) \log q \sum_{k=0}^{\infty}q^{A+\sigma+k} \left( \frac{1}{(1-q^{\theta + k)^2}} -\frac{1}{\vert 1 - q^{A + \sigma + it +k}\vert^2}\right ) 
\label{steepdescent2}
\end{equation*}
has the same sign as $\sin(t \log q)$, proving the steep-descent property. 
\end{enumerate}
\end{proof}
\begin{lem}
The kernel $K_x(W(s), W')$ has exponential decay, in the sense that there exist $N_0$,  $s_0 \geqslant 0$ and $c>0$ such that for all $s>s_0$ and $N>N_0$, 
\begin{equation*}
\vert K_x(W(s), W') \vert \leqslant \exp(-cNs).
\end{equation*} 
\label{lemmaexponentialdecay2}
\end{lem}
\begin{proof}
This Lemma is very similar with \cite[Lemma 6.10]{ferrari2013tracy} and we adapt the proof.

We first estimate the contribution of the integration along the vertical line $A +\sigma +i\R $.  For any $\varphi\in (0, \pi/4]$, 
\begin{equation*}
\lim_{s\to+\infty} \frac{\mathrm{d}}{\mathrm{d}s} \left( \Re\left[g_0(W(s))\right]\right)=\kappa >0.
\end{equation*}
Therefore, for $s$ large enough
\begin{equation*}
\Re\left[Ng_0(W(s))  + N^{1/2}g_1(W(s))\right]  > \kappa s N /2 - N^{1/2}\sigma^{1/2}x\vert \log_q(\vert \alpha + s + is\vert)\vert/2.
\end{equation*}
 Thus, one can find $N_0$ and   $s_1 \geqslant 0$ such that for all $s>s_1$ and $N>N_0$, $\exp(-Ng_0(W(s))  - N^{1/2}g_1(W(s))   ) < \exp(-\kappa  N s /4)$. As the vertical line is at a distance at least $\sigma/2$ from the poles coming from the sine, the factor $\frac{\pi}{\sin(\pi(Z-W))} $ is bounded  by $Ce^{-\pi\Im[Z]}$ for some constant $C>0$. The remaining factors in the integrand are bounded for $W\in\mathcal{C}_{A, \pi/4}$ and $Z\in\mathcal{D}_W$.

Now we estimate the contribution of the integration along the small circles $\mathcal{E}_1, \dots , \mathcal{E}_{k_W}$. It is enough to prove that each residue at the poles in $W(s)+1, \dots, W(s)+k_{W(s)}$ is at most $\exp(-cNs)$, as the number of poles is only logarithmic in $s$. Instead of reproducing word-for-word the proof of Lemma 6.10 in \cite{ferrari2013tracy}, observe that
\begin{multline}
\Re\left[g_0(W)\right]-\Re\left[g_0(W+j)\right] = \\ \left(\Re\left[g_0(W)\right]-\Re\left[f_0(W)\right]\right) + \left( \Re\left[f_0(W)\right]-\Re\left[f_0(W+j)\right]\right) + \left( \Re\left[f_0(W+j)\right]-\Re\left[g_0(W+j)\right]\right).
\label{utilisationdef-g}
\end{multline}
The sum of the first and third terms is just $-\log(q) (f-g) j$ which is positive. And by the arguments of  \cite[Lemma 6.10]{ferrari2013tracy}, for $W=W(s)$ with large $s$, there exists a constant $c'>0$ such that for $s\geqslant s_2$
\begin{equation*}
\Re\left[f_0(W)\right]-\Re\left[f_0(W+j)\right] > c'  s.
\end{equation*}
One concludes that the integrand in \ref{equationexpressionkernelgaussien} behaves like $\exp(-cNs)$ which concludes the proof.
\end{proof}

As in the case $\alpha >q^{\theta}$, we can now formulate the  Fredholm determinant on the contour $\mathcal{C}_{A, \pi/4}$ (instead of $\mathcal{C}_{A, \varphi}$ for $\varphi \in (0, \pi/4)$). As in Section \ref{sectioncriticalcase}, a small modification of the contours on a $N^{-1/2}$-neighbourhood of $A$ is needed so that the pole for $W$ in $A$ is inside the contour, and the contour $\mathcal{C}_{A, \pi/4}$ stays to the left of $\mathcal{D_W}$. 

\begin{prop}
For any fixed $\delta >0$ and $\epsilon >0$, there is an $N_1$ such that 
\begin{equation*}
\vert \det(I+K_x)_{\mathbb{L}^2(\mathcal{C}_{A, \pi/4})} - \det(I+K_{x, \delta})_{\mathbb{L}^2(\mathcal{C}_{A}^{\delta})}\vert < \epsilon
\end{equation*}
for all $N>N_1$ where $\mathcal{C}_{A}^{\delta} = \mathcal{C}_{A, \pi/4} \cap \lbrace W\; \ \vert W-A\vert \leqslant \delta \rbrace $, and 
\begin{equation}
K_{x, \delta}(W,W')= \frac{q^W \log q}{2i\pi} \int_{\mathcal{D}_W^{\delta}} \frac{\mathrm{d}Z}{q^Z - q^{W'}}\frac{\pi}{\sin(\pi(Z-W))} \frac{\exp(Ng_0(Z)+N^{1/2}g_1(Z))}{\exp(Ng_0(W)+N^{1/2}g_1(W))}\frac{\phi(Z)}{\phi(W)}
\label{equationkernelgaussienlocal}
\end{equation}
and $\mathcal{D}_{W}^{\delta} = \mathcal{D}_{W} \cap \lbrace Z\; \ \vert Z-A\vert \leqslant \delta \rbrace $
\end{prop}
\begin{proof}
Using Lemma \ref{lemmasteepdescent2} and Lemma \ref{lemmaexponentialdecay2}, one can apply exactly the same proof as in Proposition \ref{propositionlocalization}. 
We are left with proving that the contribution of all small circles in the contour $\mathcal{D}_W $ goes to zero when $N$ tends to infinity, which results from the following Lemma as in Proposition \ref{propositionlocalization}.
\begin{lem}
There exists $\eta>0$ such that   for any $W\in\mathcal{C}_{A, \pi/4}$, $\Re\left[g_0(W)-g_0(W+j)\right]>\eta$ , for all $j=1, \dots, k_W$.
\label{lemmacontributionresidusgaussien}
\end{lem}
\begin{proof}
First notice that 
\begin{equation*}
\Re\left[g_0(W)-g_0(W+j)\right]= (-\log q) \left(f(q, \theta)-g(q, \theta)\right) j + \Re
\left[f_0(W)-f_0(W+j)\right].
\end{equation*}
As $f(q, \theta)-g(q, \theta)>0$, it is enough to prove that for any $W\in\mathcal{C}_{A, \pi/4}$, $\Re\left[f_0(W)-f_0(W+j)\right]\geqslant 0$, for all $j=1, \dots, k_W$.
The proof is adapted from Lemma \ref{lemmacontributionresidus} and splits here into three parts. As before, we prove the result for the residues lying above the real axis. Let $W_A$ be the point where $\mathcal{C}_{A, \pi/4}$ and $\mathcal{C}_{\theta, \pi/2} $ intersect (above the real axis). In other words, $W_A = \log_q(q^{\theta} + i(q^{\theta} -\alpha)) = \log_q(\alpha + (1+i)(q^{\theta} -\alpha))$.

\item[\textbf{Case 1 :  $\Re\left[ q^{W+j} \right] > q^{\theta} $.}] This is the case when $W+j$ lies on the left of $\mathcal{C}_{\theta, \pi/2} $. The fact that  $\Re\left[f_0(W)-f_0(W+j)\right]\geqslant 0$ was proved in Lemma \ref{lemmacontributionresidus}.

\item[\textbf{Case 2 :  $\Re\left[ q^{W+j}\right] < q^{\theta} $ and $\Re[ W+j] \leqslant \Re \left[W_A\right] $.}] Let $W(s)$ be a parametrization of  $\mathcal{C}_{A, \pi/4}$  and $\hat{W}(t)$ a parametrization of $\mathcal{C}_{\theta, \pi/2} $. Let  $t_j$ be such that $\Re\left[ \hat{W}(t_j)\right] = \Re\left[ W+j\right]$, and $u_j$ be such that $\Im\left[ W(u_j)\right] = \Im\left[ \hat{W}(t_j)\right] $ (see Figure \ref{figurecheminresidusgaussien}). 
If $\Re\left[ W+j\right] < \Re \left[W_A\right]$, we consider the path from $W$ to $W(u_j)$ along $\mathcal{C}_{A, \pi/4}$, from $W(u_j)$ to $ \hat{W}(t_j)$ along a horizontal line, and from $\hat{W}(t_j)$ to $W+j$ along a vertical line. The fact that the horizontal segment is on the left of $\mathcal{C}_{\theta, \pi/2} $ and the vertical segment is on the right of $\mathcal{C}_{\theta, \pi/2} $ ensures, by the same arguments as in Lemma \ref{lemmacontributionresidus}, that $\Re \left[f_0\right]$ decays along this path. 

\item[\textbf{Case 3 : $\Re\left[ W_A\right] < \Re \left[W+j\right] \leqslant A$.}] Let $s_j$ be such that $\Re \left[W(s_j)\right] = \Re\left[W+j\right]$. 
We consider the path from $W$ to $W(s_j)$ along $\mathcal{C}_{A, \pi/4}$, and from $W(s_j) $  to $W+j$ along a vertical line. The fact that the vertical segment from $W(s_j) $  to $W+j$ is on the right of $\mathcal{C}_{\theta, \pi/2} $ ensures again that $\Re \left[f_0\right]$ decays along this path. 

It may also happen that $\Re\left[ W+j\right] = A+\sigma'$ for some $0<\sigma' \leqslant 2\sigma$,  but we treat this case exactly as in the end of the proof of Lemma \ref{lemmacontributionresidus}.
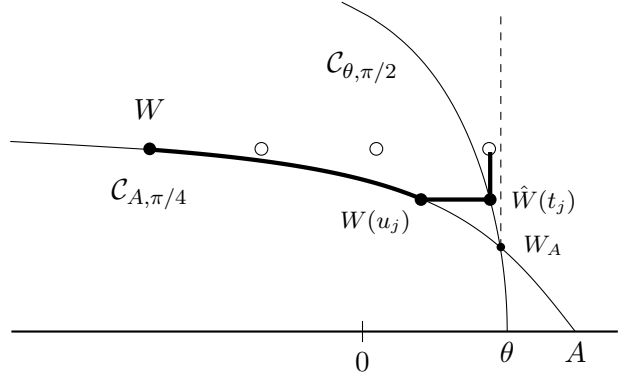
\begin{figure}
\begin{center}
\begin{tikzpicture}[scale=1.4]
\draw[thick] (-3.3, 0) -- (2.4,0);
\draw (0,0.1) -- (0,-0.1) node[anchor=north]{0};
\draw[scale=1, domain=0:1.385, smooth, variable=\t] plot({2 - tan(\t r)},{1.3*\t}); 

\draw[scale=1, domain=0:1, smooth, variable=\t] plot({1.36 - tan(\t^2 r)},{3.12*\t});

\fill[] (-2, 1.73) circle(0.06);
\draw (-0.95, 1.73) circle(0.06);
\draw (0.13, 1.73) circle(0.06);
\draw (1.19, 1.73) circle(0.06);

\fill[] (1.3, 0.8) circle(0.04);
\draw[dashed] (1.3, 0.8) -- (1.3, 3);

\fill[] (0.55, 01.25) circle(0.06);
\fill[] (1.2, 1.25) circle(0.06);
\draw[ultra thick] (0.55, 1.25) -- (1.2, 1.25);
\draw[ultra thick, scale=1, domain=0.95:1.325, smooth, variable=\t] plot({2 - tan(\t r)},{1.3*\t}); 
\draw[ultra thick] (1.2, 1.25) -- (1.2, 1.7);

\draw (-2, 2.3) node[anchor = north] {$W$};
\draw (1.7, 1) node[anchor = north] {\footnotesize{$W_A$}};
\draw (0.55, 1.25) node[anchor = north east] {\footnotesize{$W(u_j)$}};
\draw (1.3, 1.25) node[anchor = west] {\footnotesize{$\hat{W}(t_j)$}};
\draw (1.36,0) node[anchor=north] {$\theta$};
\draw (2,0) node[anchor=north] {$A$};
\draw (-2, 1.3) node{$\mathcal{C}_{A, \pi/4} $};
\draw (0, 2.5) node{$\mathcal{C}_{\theta, \pi/2} $};
\end{tikzpicture}
\end{center}
\caption{The thick line is the path from $W$ to $W+j$ with $j=3$ in the case 2 : $\Re\left[ q^{W+j}\right] < q^{\theta} $ and $\Re[ W+j] \leqslant \Re \left[W_A\right] $.}
\label{figurecheminresidusgaussien}
\end{figure}
\end{proof}

\end{proof}

For simplicity, we modify again the contours, the curves becoming straight lines. By the Cauchy theorem, this modification is authorized as soon as the endpoints  of the contours coincide.
The contour $\mathcal{C}_{A}^{\delta}$ becomes $\lbrace A + e^{i(\pi-\gamma)\sgn(y)} \vert y\vert, y \in \pm [N^{-1/2}, \delta]\rbrace\cup \lbrace N^{-1/2}e^{i t}, t\in[\gamma-\pi, \pi-\gamma]\rbrace$ where the angle $\gamma<\pi/4$ is chosen so that the endpoints coincide. 
We also consider the corresponding contour $\mathcal{C}_{N} $ for the rescaled variables $w = N^{1/2}(W-A)$. 

Similarly, the contour for the variable $Z$ becomes $\lbrace A + e^{i(\pi/2-\gamma) \sgn(y)} \vert y\vert,  y \in \pm [N^{-1/2}, \delta]\rbrace\cup \lbrace N^{-1/2}e^{i t}, t\in[-\pi/2+\gamma, \pi/2-\gamma]\rbrace$, and the constant $\sigma$ used in the definition of $\mathcal{D}_W$ is chosen so that the endpoints coincide. We also consider the corresponding contour  $\mathcal{D}_{N} $ for the rescaled variable $z=N^{1/2}(Z-A)$.
\begin{prop}
There exist $\delta' >0 $ such that for $\delta <\delta'$, 
\begin{equation*}
\lim_{N\to\infty } \vert \det(I+K_{x, \delta})_{\mathbb{L}^2(\mathcal{C}_{A}^{\delta})} -\det(I+K'_{x,N})_{\mathbb{L}^2(\mathcal{C}_N)} \vert =0
\end{equation*}
where 
\begin{equation}
K'_{x,N}(w,w') = \frac{1}{2i\pi} \int_{\mathcal{D}_N} \frac{\rm d z}{(w-z)(z-w')} \frac{\exp(\sigma z^2/2 -\sigma^{1/2}  zx)}{\exp(\sigma w^2/2 - \sigma^{1/2}wx)}\left(\frac{w}{z}\right)^k.
\label{equationkernelgaussienlimit}
\end{equation}
\label{propositionlimitfredholmgaussien}
\end{prop}
\begin{proof}
Consider the rescaled kernel 
\begin{equation*}
K_{x,\delta}^{N}(w,w') = N^{-1/2} K_{x, \delta N^{1/2}} (A+wN^{-1/2}, A+w' N^{-1/2})
\end{equation*}
where we use the new contour for $Z$ in the definition of $K_{x, \delta N^{1/2}}$, i.e. $A+N^{-1/2}\mathcal{D}_N$. By the previous discussion on the contours,  
\begin{equation*}
\det(I+K_{x, \delta})_{\mathbb{L}^2(\mathcal{C}_{A}^{\delta})} =  \det(I+K_{x, \delta}^N)_{\mathbb{L}^2(\mathcal{C}_N)}.
\end{equation*}
We first give an estimate for the exponential factor in the kernel $K_{x\delta}^N$ in Equation (\ref{equationkernelgaussienlocal}). By Taylor approximation, there exists $C_{g_0}$ such that for $ \vert Z-A \vert < A$, 
\begin{equation*}
\vert g_0(Z) - g_0(A) - \frac{\sigma}{2}(Z-A)^2\vert < C_{g_0} \vert Z-A \vert^3 
\end{equation*}
and since $g_1'(A) = -\sigma^{1/2} x $, there exists $C_{g_1}$ such that 
\begin{equation*}
\vert g_1(Z)- g_1(A)  +\sigma^{1/2} x (Z-A)\vert < C_{g_1} \vert Z-A \vert^2.  
\end{equation*}
The argument in the exponential in (\ref{equationkernelgaussienlocal}) is $g(Z, W, N):= N(g_0(Z) - g_0(W)) + N^{1/2}(g_1(Z) - g_1(W))$. Let us denote $g^{\mathrm{lim}}(z, w) = \frac{\sigma}{2}(z^2-w^2) - \sigma^{1/2}x(z-w) $. Using the Taylor expansions above and using the change of variables $Z = A+zN^{-1/2}$, and likewise for $W$ and $W'$, one has 
\begin{equation}
\left| g(Z, W, N) - g^{\mathrm{lim}}(z, w) \right| <N^{-1/2}\left( C_{g_0 }(\vert z\vert^3 +\vert w\vert^3 )+C_{g_1 }(\vert z\vert^2 +\vert w\vert^2 ) \right).
\label{equationestimeeexponentiellegaussien1}
\end{equation}
For $z\in \mathcal{D}_N$ and $w\in\mathcal{C}_N$, the last inequality rewrites 
\begin{equation}
\left| g(Z, W, N) - g^{\mathrm{lim}}(z, w) \right| <\delta\left( C_{g_0 }(\vert z\vert^2 +\vert w\vert^2 )+C_{g_1 }(\vert z\vert+\vert w\vert ) \right).
\label{equationestimeeexponentiellegaussien2}
\end{equation}

Now we estimate the remaining factors in the integrand in (\ref{equationkernelgaussienlocal}). Let us denote
\begin{equation*}
G(Z,W,W'):= \frac{N^{-1/2}}{q^Z-q^{W'}} \frac{\pi}{\sin(\pi(Z-W))}\frac{\phi(Z)}{\phi(W)}, 
\end{equation*}
and the remaining factors in the integrand in (\ref{equationkernelgaussienlimit}) as 
\begin{equation*}
G^{\mathrm{lim}}(z,w,w'):= \frac{1}{z-w'}\frac{1}{z-w} \left(\frac{w}{z}\right)^k.
\end{equation*}

We prove first that for any $w,w' \in \mathcal{C}_N$, $K_{x, \delta}^N(w,w') - K'_{x,N}(w,w') $ goes to zero when $N$ goes to infinity. The error can be estimated by 
\begin{multline}
\vert K_{x, \delta}^N(w,w') - K'_{x,N}(w,w') \vert \leqslant  \int_{\mathcal{D}_N}\mathrm{d}z \exp(g^{\mathrm{lim}}) \vert G(Z, W, W' ) \vert \left| \exp(g - g^{\mathrm{lim}})-1\right| \\
+ \int_{\mathcal{D}_N}\mathrm{d}z \exp(g^{\mathrm{lim}}) \left| G - G^{\mathrm{lim}} \right|, 
\label{equationestimeeerreurgaussien}
\end{multline}
where we have omitted the arguments of the functions $g(Z, W, N)$, $g^{\mathrm{lim}}(z,w)$, $G(Z,W,W')$, $G^{\mathrm{lim}}(z,w,w')$, with $Z=A + zN^{-1/2}$ as before, and likewise for $W, W'$. By (\ref{equationestimeeexponentiellegaussien1}) and (\ref{equationestimeeexponentiellegaussien2}) and the inequality $\vert \exp(x) - 1\vert \leqslant \vert x\vert \exp(\vert x\vert)$, we have 
\begin{equation*}
\vert \exp(g - g^{\mathrm{lim}})-1\vert  < N^{-1/2} P(\vert z\vert, \vert w \vert) \exp\left(\delta \left(C_{g_0} (\vert z\vert^2+\vert w\vert^2)+C_{g_1} (\vert z\vert+\vert w\vert) \right) \right), 
\end{equation*}
where $P$ is the polynomial $P(X,Y) = C_{g_0} (X^3+Y^3)+C_{g_1} (X^2+Y^2)$. Hence for $\delta$ small enough, the first integral in (\ref{equationestimeeerreurgaussien}) has quadratic exponential decay (due to the decay of $\exp(g^{\mathrm{lim}})$). Thus, by dominated convergence,  the first integral in (\ref{equationestimeeerreurgaussien}) goes to zero as $N$ goes to infinity by dominated convergence. Using estimate (\ref{equationboundphicritical}) in Lemma \ref{lemmapointwiselimitcritical}, the second integral in (\ref{equationestimeeerreurgaussien}) also goes to zero since
\begin{multline*}
\left| G(A +zN^{-1/2}, A +wN^{-1/2}, A +w'N^{-1/2})- G^{\mathrm{lim}}(z,w,w')\right| < \\ N^{-1/3} Q(\vert z \vert, \vert w \vert, \vert w' \vert) G^{\mathrm{lim}}(z,w,w'), 
\end{multline*}
for some polynomial $Q$.

 Moreover, the estimates in right-hand-sides of Equations (\ref{equationestimeeexponentiellegaussien2}) and (\ref{equationboundphicritical}) show that there exists a constant $C>0$ independent of $N$ such that for all $w, w'\in\mathcal{C}_N$,
 \begin{equation*}
\vert K_{x, \delta}^N (w,w')\vert < C  \exp\left(-\sigma/2 w^2 +C_{g_0}\delta \vert w \vert^2 +\sigma^{1/2}xw + C_{g_1}\delta \vert w\vert\right).
 \end{equation*}
 Hence for $\delta$ small enough, Hadamard's bound yields
\begin{equation*}
\left|\det\left(K_{x, \delta}^N (w_i,w_j)_{1\leqslant i,j\leqslant n} \right ) \right| \leqslant n^{n/2} C^n \prod_{i=1}^n e^{-\sigma/4\Re \left[w_i^2\right]}.
\end{equation*}
It follows that  the Fredholm determinant expansion,
\begin{equation*}
\det(I+K_{x, \delta}^N)_{\mathbb{L}^2(\mathcal{C}_N)} = \sum_{n=0}^{\infty} \frac{1}{n!} \int_{\mathcal{C}_N}\mathrm{d}w_1 \dots \int_{\mathcal{C}_N}\mathrm{d}w_n \det\left(K_{x, \delta}^N (w_i,w_j)_{1\leqslant i,j\leqslant n} \right ),
\end{equation*}
is absolutely integrable and summable. The conclusion of the Proposition follows by dominated convergence.
\end{proof}
Finally, since the integrand has quadratic exponential decay along the contours $\mathcal{C}_{\infty}$ and  $\mathcal{D}_{\infty}$, dominated convergence, again,  yields 
\begin{equation*}
\det(I+ K'_{x,N})_{\mathbb{L}^2(\mathcal{C}_N)} \underset{N\to\infty}{\longrightarrow} \det(I+ K'_{x,\infty})_{\mathbb{L}^2(\mathcal{C}_\infty)}.
\end{equation*}
The third part of Theorem \ref{theomainresult} now follows from a reformulation of the Fredholm determinant achieved in the following Proposition.
\begin{prop}
\begin{equation*}
\det(I + K'_{x,\infty})_{\mathbb{L}^2(\mathcal{C}_\infty)} = G_k(x)
\end{equation*}
where $G_k$ is defined in definition \ref{defdistributions}.
\end{prop} 
\begin{proof}
Using the identity,
\begin{equation*}
\frac{1}{z-w} = \int_{\R_+} \mathrm{d}\lambda e^{-\lambda (z-w)}, 
\end{equation*}
valid when $\Re[z-w]>0$, the operator $K'_{x,\infty} $ can be factorized. $K'_{x,\infty}(w,w')  = -\left(AB \right)(w,w')$ where $A : \mathbb{L}^2(\R_+) \rightarrow\mathbb{L}^2(\mathcal{C}_\infty) $ and $B :  \mathbb{L}^2(\mathcal{C}_\infty)\rightarrow \mathbb{L}^2(\R_+)$  are Hilbert–Schmidt operators having kernels 
\begin{equation*}
A(w,\lambda) = e^{-w^2/2+w(x+\lambda)} w^k \hspace{1cm}\text{and} \hspace{1cm} B(\lambda, w') = \frac{1}{2i\pi}\int_{ \mathcal{D}_{\infty}} \frac{\mathrm{d}z}{z^k} \frac{e^{z^2/2-z(x+\lambda)}}{z-w'}.
\end{equation*}
We also have 
\begin{equation*}
BA(\lambda, \lambda') = \frac{1}{2i\pi} \int_{\mathcal{C}_{\infty}} \mathrm{d}w B(\lambda, w)A(w, \lambda')  = H_k(\lambda+x, \lambda'+x).
\end{equation*}
Since $\det(I-AB)_{\mathbb{L}^2(\mathcal{C}_\infty)} = \det(I-BA)_{\mathbb{L}^2(\R_+)} = G_k(x)$, we get the result.
\end{proof}

\section*{Ackowledgements}
The author would like to thank his thesis advisor Sandrine P\'ech\'e for her support through this work and  very useful comments on an earlier version of this paper.

\bibliographystyle{plain}
\bibliography{qTASEPslow.bib}
\end{document}